\documentclass[english,10pt]{amsart}

\usepackage[english]{babel}

\usepackage{amscd,amssymb,amsfonts,amsmath}
\usepackage{tikz-cd}

\usepackage{bm}
\usepackage{graphics}
\usepackage{epsfig}
\usepackage{mathrsfs}
\usepackage{xcolor}
\DeclareMathAlphabet{\mathscrbf}{OMS}{mdugm}{b}{n}

\definecolor{violet}{rgb}{0.0,0.2,0.7}
\definecolor{rouge2}{rgb}{0.8,0.0,0.2}
\usepackage{hyperref}
\hypersetup{
    unicode=false,          % non-Latin characters in Acrobat’s bookmarks
    pdftoolbar=true,        % show Acrobat’s toolbar?
    pdfmenubar=true,        % show Acrobat’s menu?
    pdffitwindow=false,     % window fit to page when opened
    pdfstartview={FitH},    % fits the width of the page to the window
    pdftitle={},    % title
    pdfauthor={},     % author
    colorlinks=true,       % false: boxed links; true: colored links
   linkcolor=violet,          % color of internal links
    citecolor=rouge2,        % color of links to bibliography
    filecolor=black,      % color of file links
    urlcolor=cyan}           % color of external links

\usepackage[text={6.7in,9.2in},centering]{geometry}

\makeatletter
\renewcommand\subsection{\@startsection{subsection}{2}%
  \z@{.5\linespacing\@plus.7\linespacing}{-.5em}%
  {\normalfont\sffamily}}  
\makeatother

\renewcommand{\div}{\textup{div}}

\renewcommand{\phi}{\varphi}
\newcommand{\into}{\hookrightarrow}
\newcommand{\map}{\dashrightarrow}

\renewcommand{\le}{\leqslant}
\renewcommand{\ge}{\geqslant}

\newcommand{\mult}{\textup{mult}}

\newcommand{\sE}{\mathscr{E}}

\newcommand{\sG}{\mathscr{G}}
\newcommand{\sH}{\mathscr{H}}

\newcommand{\sL}{\mathscr{L}}

\newcommand{\sN}{\mathscr{N}}
\newcommand{\sO}{\mathscr{O}}

\newcommand{\Der}{\textup{Der}}

\newtheorem{thm}{Theorem}[section]

\newtheorem{question}[thm]{Question}
\newtheorem{lemma}[thm]{Lemma}
\newtheorem{cor}[thm]{Corollary}

\newtheorem{prop}[thm]{Proposition}

\newtheorem*{thm*}{Theorem}
\theoremstyle{definition}
\newtheorem{defn}[thm]{Definition}

\newtheorem{defn-thm}[thm]{Definition-Theorem} 
\newtheorem{defn-lemma}[thm]{Definition-Lemma}

\theoremstyle{remark}

\newtheorem{fact}[thm]{Fact}

\newtheorem*{not-and-def}{Notation and definitions}
\newtheorem{assumption}[thm]{Assumption} 
\newtheorem{rem}[thm]{Remark}

\numberwithin{equation}{section}

\def\factor#1.#2.{\left. \raise 2pt\hbox{$#1$} \right/\hskip -2pt\raise -2pt\hbox{$#2$}}

\begin{document} 

\title[Numerical characterization of some toric fiber bundles]{Numerical characterization of some toric fiber bundles}

\author{St\'ephane \textsc{Druel}}

\address{St\'ephane Druel: Univ Lyon, CNRS, Universit\'e Claude Bernard Lyon 1, UMR 5208, Institut Camille Jordan, F-69622 Villeurbanne, France} 

\email{stephane.druel@math.cnrs.fr}

\author{Federico Lo Bianco}

\address{Federico Lo Bianco: Univ Lyon, CNRS, Universit\'e Claude Bernard Lyon 1, UMR 5208, Institut Camille Jordan, F-69622 Villeurbanne, France} 

\email{lobianco@math.univ-lyon1.fr}


\begin{abstract}
Given a complex projective manifold $X$ and a divisor $D$ with normal crossings, we say that the logarithmic tangent bundle
$T_X(-\textup{log}\, D)$ is R-flat if its pull-back to the normalization of any rational curve contained in $X$ is the trivial vector bundle.
If moreover $-(K_X+D)$ is nef, then the log canonical divisor $K_X+D$ is torsion and the maximally rationally chain connected fibration turns out to be a smooth locally trivial fibration with typical fiber $F$ being a toric variety with boundary divisor $D_{|F}$.
\end{abstract}

\maketitle

%\tableofcontents
{\small\tableofcontents}

\section{Introduction}

Let $X$ be a smooth projective algebraic variety $X$ over the field of complex numbers and let $D$ be a divisor with normal crossings. The structure of pairs $(X,D)$ with trivial logarithmic tangent bundle $T_X(-\textup{log}\, D)$ is well understood by a result of Winkelmann (see \cite{winkelmann_log_trivial}). They are called semiabelic varieties. 
The simplest examples of pairs $(X,D)$ with trivial logarithmic tangent bundle $T_X(-\textup{log}\, D)$ are pairs $(A,0)$ where $A$ is an abelian variety, and pairs $(X,D)$ where $X$ is a smooth toric variety with boundary divisor $D$.
If $X$ is smooth projective semiabelic variety, then the algebraic group $G:=\textup{Aut}^{0}(X,D)$ is a semiabelian group which acts on $X$ with finitely many orbits. Moreover, the $G$-orbits in $X$ are exactly the strata defined by $D$. As a consequence, the albanese map is a smooth locally trivial fibration with typical fiber $F$ being a toric variety with boundary divisor $D_{|F}$. In particular, semiabelic varieties are toric fiber bundles over abelian varieties. 

However, from the point of view of birational classification of algebraic varieties, it is more natural to consider the case where the logarithmic tangent bundle $T_X(-\textup{log}\, D)$ is numerically flat (see Definition \ref{defn: numerically flat} for this notion). If $D=0$, then $X$ is covered by an abelian variety, as a classical consequence of Yau's theorem on the existence of a K\"ahler-Einstein metric. In the present paper, we obtain a numerical characterization of a class of toric fiber bundles containing pairs $(X,D)$ with numerically flat logarithmic tangent bundle $T_X(-\textup{log}\, D)$ (see Corollary \ref{cor:main_intro_2}).  

Our main result is the following. A vector bundle $\sE$ on a projective variety $X$ is called R-flat if $\nu^*\sE$ is the trivial vector bundle for any morphism $\nu\colon\mathbb{P}^1 \to X$ (see paragraph \ref{R_flat_vector_bundles}).

\begin{thm}\label{thm:main_intro}
Let $(X,D)$ be a log smooth reduced pair with $X$ projective. Suppose that $-(K_X+D)$ is nef and that 
$T_X(-\textup{log}\, D)$ is R-flat.
Then there exist a smooth projective variety $T$ with $K_T\equiv 0$ as well as a smooth morphism with connected fibers 
$a\colon X \to T$. The fibration $(X,D)\to T$ is locally trivial for the analytic topology and any fiber $F$ of the map $a$ is a smooth toric variety with boundary divisor $D_{|F}$. Moreover, $T$ contains no rational curve.
\end{thm}

In fact, a more general statement is true (see Theorem \ref{thm:main}) but its formulation is somewhat involved.

\begin{rem}
Let $(X,D)\to T$ is a toric fiber bundle over a projective manifold $T$ with $K_T\equiv 0$. Suppose in addition that $T$ contains no rational curve. Then $T_X(-\textup{log}\, D)$ is obviously R-flat. Moreover, $K_X+D \equiv 0$ by Theorem \ref{thm:cbf}. 
\end{rem}

\begin{rem}\label{rem:torus_quotient}
In the setup of Theorem \ref{thm:main_intro}, we expect that $T$ is a torus quotient. Indeed, by the Beauville-Bogomolov decomposition theorem, $T$ admits a finite \'etale cover that decomposes into the product of an abelian variety and a simply-connected Calabi-Yau manifold. On the other hand, a folklore conjecture asserts that any projective Calabi-Yau manifold contains a rational curve. 
\end{rem}

This motivates the following question.

\begin{question}\label{question:torus_quotient}
Let $(X,D)$ be a log smooth reduced pair with $X$ projective. Suppose that $-(K_X+D)$ is nef and that 
$T_X(-\textup{log}\, D)$ is R-flat. Is $X$ a toric fiber bundle over an \'etale quotient of an abelian variety?
\end{question}

The following results are rather easy consequences of Theorem \ref{thm:main_intro} above.

\begin{cor}\label{cor:main_intro_1}
Let $(X,D)$ be a log smooth reduced pair with $X$ projective. Suppose that $-(K_X+D)$ is nef and that 
$T_X(-\textup{log}\, D)$ is R-flat. Suppose in addition that $X$ is simply-connected and that $h^{p,0}(X)=0$ for all
$1 \le p \le \dim X$. Then $X$ is a smooth toric variety with boundary divisor $D$. 
\end{cor}

The next question is a special case of Question \ref{question:torus_quotient} above.

\begin{question}
Let $(X,D)$ be a log smooth reduced pair with $X$ projective. Suppose that $-(K_X+D)$ is nef and that 
$T_X(-\textup{log}\, D)$ is R-flat. Suppose in addition that $X$ is simply-connected. Is $X$ a smooth toric variety with boundary divisor $D$?
\end{question}

\begin{cor}\label{cor:main_intro_2}
Let $(X,D)$ be a log smooth reduced pair with $X$ projective. Suppose that $T_X(-\textup{log}\, D)$ is numerically flat.
Then there is a smooth morphism $a\colon X \to T$ with connected fibers onto a torus quotient $T$. 
The fibration $(X,D)\to T$ is locally trivial for the analytic topology and any fiber $F$ of the map $a$ is a smooth toric variety with boundary divisor $D_{|F}$.
\end{cor}

The next result says in particular that Corollary \ref{cor:main_intro_2} applies to pairs with flat logarithmic tangent sheaf. 

\begin{prop}\label{prop:main_intro_3}
Let $(X,D)$ be a log smooth reduced pair with $X$ projective. If $T_X(-\textup{log}\, D)$ admits a holomorphic connection, then 
it is numerically flat.
\end{prop}

The proof of Theorem \ref{thm:main_intro} relies in part on a descent theorem for vector bundles which is of independent interest. 
Our result extends \cite[Theorem 1.1]{biswas_vb_rc} to the relative setting.

\begin{thm}\label{thm:main_descent_intro}
Let $f \colon X \to Y$ be a projective morphism with connected fibers of normal, quasi-projective varieties. 
Suppose that $X$ is smooth and that $Y$ is klt. Suppose in addition that $f$ has rationally chain connected fibers. Let $\sE$ be a locally free, $f$-relatively R-flat sheaf on $X$. Then there exists a locally free sheaf $\sG$ on $Y$ such that $\sE\cong f^* \sG$.
\end{thm}

\subsection*{Structure of the paper} 
Section 2 gathers notation, global conventions, and known results that will be used throughout the paper. We also establish some facts. In particular, we establish a number of properties of R-flat vector bundles.
In section 3, we prove a canonical bundle formula for generically isotrivial fibrations.
Section 4 is devoted to the proof of Theorem \ref{thm:main_descent_intro}. Section 5 prepares for the proof of the main results. 
With these preparations at hand, the proof of Theorem \ref{thm:main_intro} as well as the proofs of Corollaries \ref{cor:main_intro_1} and \ref{cor:main_intro_2} and Proposition \ref{prop:main_intro_3}, which we give in Section 6, become reasonably short.

\subsection*{Acknowledgements} The authors would like to thank Enrica Floris for discussions concerning the canonical bundle formula.
We would also like to thank Beno\^it Claudon for helpful conversations.

The first author was partially supported by the ERC project ALKAGE (ERC grant Nr 670846), the CAPES-COFECUB project Ma932/19 and the ANR project Foliage (ANR grant Nr ANR-16-CE40-0008-01).

The second author was supported by the LABEX MILYON (ANR-10-LABX-0070) of Universit\'e de Lyon, within the program \textit{Investissements d'Avenir} (ANR-11-IDEX- 0007) operated by the ANR.

\section{Notation, convention and used facts}

\subsection{Global conventions} Throughout the paper, all varieties are assumed to be defined over the field of complex numbers.
Given a variety $X$, we denote by $X_\textup{reg}$ its smooth locus.
 
\subsection{Pull-back of Weil divisors}\label{definition:pull-back}
Let $\psi\colon X \to Y$ be a dominant equidimensional morphism of normal varieties, and let $D$ be a Weil $\mathbb{Q}$-divisor on $Y$. The \textit{pull-back $\psi^*D$ of $D$} is defined as follows. We define 
$\psi^*D$ to be the unique $\mathbb{Q}$-divisor on $X$ whose restriction to 
$\psi^{-1}(Y_{\textup{reg}})$ is $(\psi_{|\psi^{-1}(Y_{\textup{reg}})})^*(D_{|Y_{\textup{reg}}})$. This construction agrees with the usual pull-back if $D$ is $\mathbb{Q}$-Cartier.
  
\subsection{Exceptional divisor} We will need the following definition.

\begin{defn}\label{defn:exceptional}
Let $f\colon X \map Y$ be a dominant rational map of normal varieties. Suppose in addition that $Y$ is projective.
A prime divisor $Q$ on $X$ is called \textit{$f$-exceptional} if $\textup{codim}_Y f(Q) \ge 2$.
\end{defn}

\begin{rem}
Setup as in Definition \ref{defn:exceptional}. The image $f(Q)$ of $Q$ is well-defined since $Y$ is projective by assumption.
\end{rem}

\subsection{Projective space bundle}
If $\sE$ is a locally free sheaf of finite rank on a variety $X$, 
we denote by $\mathbb{P}_X(\sE)$ the variety $\textup{Proj}_X\big(\textup{S}^\bullet\sE\big)$,
and by $\sO_{\mathbb{P}_X(\sE)}(1)$ its tautological line bundle.

\subsection{Reflexive hull}
Given a normal variety $X$ and a coherent sheaf $\sE$ on $X$, write $\det\sE:=(\Lambda^{\textup{rank} \,\sE}\sE)^{**}$.
Given any morphism $f \colon Y \to X$ of normal varieties, write $f^{[*]}\sE:=(f^*\sE)^{**}.$ 

\subsection{Singularities of pairs}

A \textit{pair} $(X,D)$ consists of a normal quasi-projective variety $X$ and a (not necessarily effective) $\mathbb{Q}$-divisor $D$ on $X$.
A \textit{reduced pair} is a pair $(X,D)$ such that $D$ is effective and reduced.
We will use the notions of canonical, klt, and log canonical singularities for pairs
without further explanation or comment and simply refer to \cite{kollar97}
for a discussion and for their precise definitions.

The following elementary fact will be used throughout the paper (see \cite[Proposition 3.16]{kollar97}).

\begin{fact}\label{fact:quasi_etale_cover_and_singularities}
Let $\gamma\colon X_1 \to X$ be a finite cover between normal complex varieties. Let $D$ be a $\mathbb{Q}$-divisor on $X$, and set
$D_1:=\gamma^*(K_X+D)-K_{X_1}$. Then $(X,D)$ is klt (resp. log canonical) if and only $(X_1,D_1)$ is klt (resp. log canonical).
\end{fact}

We will also need the following definition.

\begin{defn}
A normal, quasi-projective variety $X$ is said to be of \textit{klt type} if there exists an effective $\mathbb{Q}$-divisor $D$ on $X$ such that $(X,D)$ is klt.
\end{defn}

\begin{rem}
If $X$ is of klt type and $\mathbb{Q}$-factorial, then $X$ has klt singularities.
\end{rem}

\begin{lemma}\label{lemma:klt_singularities_finite_cover}
Let $X$ be a normal projective variety and let $D$ be a reduced effective divisor on $X$ such that $K_X+D$ is $\mathbb{Q}$-Cartier.
Let $\gamma\colon X_1 \to X$ be a finite cover of normal projective varieties. Suppose that $\gamma$ is quasi-\'etale over $X \setminus \textup{Supp}\,D$.
Then the following holds.
\begin{enumerate}
\item The divisor $D_{1,\varepsilon}:=\gamma^*\big(K_X+(1-\varepsilon)D\big)-K_{X_1}$ is effective if $0\le \varepsilon \ll 1$. The divisor 
$D_1:=\gamma^*(K_X+D)-K_{X_1}$ is (effective and) reduced, and $\gamma^{-1}(\textup{Supp}\,D)\subseteq \textup{Supp}\,D_1$. 
\item Suppose in addition that $(X,D)$ is log canonical and that $X$ is $\mathbb{Q}$-factorial with klt singularities. Then $X_1$ is of klt type.
\end{enumerate}
\end{lemma}

\begin{proof}
Let $Q$ be a prime divisor on $X_1$ and let $m$ be the ramification index of $\gamma$ along $Q$. Set $P:=\gamma(Q)$. 
A straightforward local computation then shows that
$$\textup{mult}_Q\,D_{1,\varepsilon} = -(m-1)+m(1-\varepsilon)\textup{mult}_P\, D.$$
By assumption, any irreducible component of the branch locus of $\gamma$ which has codimension $1$ in $X$ is contained in the support of $D$.
Thus, if $P$ is not contained in the support of $D$, then we must have $m=1$. Item (1) follows easily.

Suppose from now on that $(X,D)$ is log canonical and that $X$ is $\mathbb{Q}$-factorial with klt singularities. Then 
$\big(X,(1-\varepsilon) D\big)$ is klt for any $0<\varepsilon \le 1$. This implies that the pair $(X_1,D_{1,\varepsilon})$ is klt as well. Item (2) now follows from Item (1), completing the proof of the lemma. 
\end{proof}

\subsection{Logarithmic differential forms}
Let $X$ be a smooth variety, and let
$D \subset X$ a divisor with simple normal crossings. Let  
$$T_X(- \textup{log}\, D) \subseteq T_X = \Der_{\mathbb{C}}(\sO_X)$$ 
be the subsheaf consisting of those derivations that preserve the ideal sheaf $\sO_X(-D)$. One easily checks that the 
\textit{logarithmic tangent sheaf} $T_X(-\textup{log}\,D)$ is a locally free sheaf of Lie subalgebras of $T_X$, having the same restriction 
to $X \setminus D$.
Set $n:=\dim X$. If $D$ is defined at $x$ by the equation $x_1\cdots x_k=0$, where $x_1,\ldots,x_k$ form part of a regular system of parameters $(x_1,\ldots,x_n)$ of the local ring $\sO_{X,x}$ of $X$ at $x$, then a local basis of $T_X(- \textup{log}\,D)$ (after localization at $x$) consists of
$$x_1 \partial_1, \ldots, x_k \partial_k, \partial_{k+1}, \ldots, \partial_n,$$ 
where $(\partial_1, \ldots, \partial_n)$ is the local basis of $T_X$
dual to the local basis $(dx_1, \ldots, dx_n)$ of $\Omega^1_X$.

A local computation shows that $T_X(-\textup{log}\, D)$ can be identified with the subsheaf of $T_X$ containing those vector fields that are tangent to $D$ at smooth points of $D$.

The dual of $T_X(-\textup{log}\, D)$ is the sheaf $\Omega^1_X( \textup{log}\,D)$ of logarithmic differential $1$-forms. More generally,
if $1 \le p \le n$, then $\Omega^p_X( \textup{log}\,D):=\wedge^p \Omega^1_X( \textup{log}\,D)$ is the sheaf \textit{logarithmic differential $p$-forms}, that is, of rational $p$-forms $\alpha$ on $X$ such that $\alpha$ and $d\alpha$ have at most simple poles along $D$. 
The top exterior power $\det \Omega^1_X( \textup{log}\,D)=\Omega^n_X( \textup{log}\,D)$ is the invertible sheaf $\sO_X(K_X+D)$, where $K_X$ denotes a canonical divisor.

\subsection{Reflexive (logarithmic) differentials forms} \label{subsection:pull-back}
Given a normal variety $X$, we denote the sheaf of K\"{a}hler differentials by
$\Omega^1_X$. 
If $0 \le p \le \dim X$ is any integer, write
$\Omega_X^{[p]}:=(\Omega_X^p)^{**}$.
The tangent sheaf $(\Omega_X^1)^*$ will be denoted by $T_X$.

If $D$ is a reduced effective divisor on $X$ we denote by 
$(X,D)_{\textup{reg}}$ the open set where $(X,D)$ is log smooth. 
If $1 \le p \le \dim X$ is any integer, we write $\Omega_X^{[p]}(\textup{log}\, D)$ for
the reflexive sheaf on $X$ whose restriction to $U:=(X,D)_{\textup{reg}}$ is the sheaf of logarithmic differential forms 
$\Omega_U^{p}(\textup{log}\, D_{|U})$. We will refer to it as the sheaf of \textit{reflexive logarithmic $p$-forms}.
Suppose that $X$ is smooth and let $t$ be a defining equation for $D$ on some open set $X^\circ$.
Let $\alpha$ be a rational $p$-form on $X$. Then $\alpha$ is a reflexive logarithmic $p$-form on $X^\circ$
if and only if $t\alpha$ and $td\alpha$ are regular on $X^\circ$ (see \cite{saito}).

The dual of $\Omega^1_X( \textup{log}\,D)$ is \textit{logarithmic tangent sheaf} $T_X(-\textup{log}\, D)$.

\begin{lemma}\label{lemma:pull_back_cover}
Let $\gamma\colon X_1 \to X$ be a finite cover between normal varieties, and let $D$ be a reduced effective divisor on $X$.
Let $1\le p \le \dim X$ be any integer.
\begin{enumerate}
\item If $D_1$ is a reduced effective divisor on $X_1$ such that $\gamma^{-1}(\textup{Supp}\,D)\subseteq \textup{Supp}\,D_1$,
then the standard pull-back map of K\"ahler differentials induces an injective map of reflexive sheaves
$$\gamma^{[*]}\Omega_X^{[p]}(\textup{log}\, D) \into \Omega_{X_1}^{[p]}(\textup{log}\, D_1).$$
\item Suppose that $\gamma$ is quasi-\'etale over $X \setminus \textup{Supp}\,D$ and set $D_1:=\gamma^*(K_X+D)-K_{X_1}$. Then $D_1$ is reduced and effective. Moreover, the standard pull-back map of K\"ahler differentials induces an isomorphism
$$\gamma^{[*]}\Omega_X^{[p]}(\textup{log}\, D) \cong \Omega_{X_1}^{[p]}(\textup{log}\, D_1).$$
\end{enumerate}
\end{lemma}
\begin{proof}
Let $D_1$ be a reduced effective divisor on $X_1$ such that $\gamma^{-1}(\textup{Supp}\,D)\subseteq \textup{Supp}\,D_1$.
Let $U \subseteq (X,D)_{\textup{reg}}$ be an open set with complement of codimension at least $2$ such that $\gamma^{-1}(U) \subseteq (X_1,D_1)_{\textup{reg}}$. By \cite[Lemma 0.2.13]{kmm} applied to the restriction of $\gamma$ to $\gamma^{-1}(U)$, the standard pull-back map of K\"ahler differentials induces an injective map of locally free sheaves
$$(\gamma_{|\gamma^{-1}(U)})^*\Omega_U^{p}(\textup{log}\, D_{|U}) \into \Omega_{\gamma^{-1}(U)}^{p}(\textup{log}\, {D_1}_{|\gamma^{-1}(U)}).$$ This easily implies that there is an injective morphism of reflexive sheaves
$$\gamma^{[*]}\Omega_X^{[p]}(\textup{log}\, D) \into \Omega_{X_1}^{[p]}(\textup{log}\, D_1).$$

Suppose now that $\gamma$ is quasi-\'etale over $X \setminus \textup{Supp}\,D$ and set $D_1:=\gamma^*(K_X+D)-K_{X_1}$. 
By Lemma \ref{lemma:klt_singularities_finite_cover}, $D_1$ is effective and reduced, and $\gamma^{-1}(\textup{Supp}\,D)\subseteq \textup{Supp}\,D_1$. A straightforward local computation then shows that the above map yields an isomorphism
$$\gamma^{[*]}\Omega_X^{[p]}(\textup{log}\, D) \cong \Omega_{X_1}^{[p]}(\textup{log}\, D_1),$$
finishing the proof of the lemma.
\end{proof}

\begin{lemma}\label{lemma:pull_back_fibration}
Let $f \colon X \to Y$ be surjective morphism of normal varieties, and let $D$ and $B$ be reduced effective divisors on $X$ and $Y$ respectively.
Suppose that $\textup{Supp}\,D$ contains all codimension $1$ irreducible components of $f^{-1}(\textup{Supp}\,B)$ which are not $f$-exceptional. Then the standard pull-back map of K\"ahler differentials induces an injective map of reflexive sheaves
$$\Omega_Y^{[p]}(\textup{log}\, B) \to \big(f_*\Omega_X^{[p]}(\textup{log}\, D)\big)^{**}$$
for any integer $1 \le p \le \dim X$.
\end{lemma}

\begin{proof}
Let $C$ be the reduced effective divisor on $X$ whose support is the union of the codimension $1$ irreducible components of $f^{-1}\big(\textup{Supp}\,B)$ which are not $f$-exceptional. Let $V\subseteq (Y,B)_{\textup{reg}}$ be the complement of the images of $f$-exceptional divisors contained in $f^{-1}\big(\textup{Supp}\,B)$, and set $U:=(X,C)_{\textup{reg}} \cap f^{-1}(V)$. By assumption, $\textup{Supp}\,C \subseteq \textup{Supp}\,D$. 
By \cite[Chapitre 2, Proposition 3.2]{deligne_equ_diff} applied to the restriction of $f$ to $U$, the standard pull-back map of K\"ahler differentials induces an injective map of sheaves
$$ \Omega_V^{p}(\textup{log}\, B_{|V}) \to (f_{|U})_*\Omega_U^{p}(\textup{log}\, C_{|U})$$
for any $1 \le p \le \dim Y$. On the other hand, since $D-C$ is effective by assumption, we have a natural map
$$(f_{|U})_*\Omega_U^{p}(\textup{log}\, C_{|U}) \to (f_{|U})_*\Omega_U^{p}(\textup{log}\, D_{|U})=(f_{|f^{-1}(V)})_*\Omega_{f^{-1}(V)}^{[p]}(\textup{log}\, D_{|f^{-1}(V)}).$$
This easily implies that there is a morphism
$$\Omega_Y^{[p]}(\textup{log}\, B) \to \big(f_*\Omega_X^{[p]}(\textup{log}\, D)\big)^{**},$$
finishing the proof of the lemma.
\end{proof}

\subsection{Resolution of singularities}

We will consider a suitable resolution of singularities of a given variety whose existence is guaranteed by the following theorem.

\begin{thm}[{\cite[Corollary 4.7]{greb_kebekus_kovacs10}}]\label{thm:canonical_resolution}
Let $X$ be a normal variety and let $D$ be a reduced effective divisor on $X$. Then there exists a log resolution
$\beta\colon Y\to X$ of $(X,D)$ such that $\beta_* T_Y(-\textup{log}\,E)\cong T_X(-\textup{log}\,D)$, 
where $E$ is the largest reduced divisor contained in $\beta^{-1}\big(\textup{Supp}\,D\big)$. 
\end{thm}

We call a resolution $\beta$ as in Theorem \ref{thm:canonical_resolution} a \textit{canonical resolution} of the pair $(X,D)$.
In the course of the proof of our main result, we will need the following observation (see \cite[Corollary 3.5]{bergner_log_ZL} for a somewhat related result).

\begin{lemma}\label{lemma:canonical_resolution}
Let $(X,D)$ be a log canonical pair with $D$ effective and reduced. Let $\beta\colon Y \to X$ be a canonical resolution of $(X,D)$ and let $E$ be the largest reduced divisor contained in $\beta^{-1}\big(\textup{Supp}\,D\big)$.
Suppose that $\Omega_X^{[1]}(\textup{log}\,D)$ is locally free. Then 
$\beta^* \Omega_X^{[1]}(\textup{log}\,D)\cong \Omega_Y^{1}(\textup{log}\,E)$. In particular, $\Omega_Y^{1}(\textup{log}\,E)$ is locally free. 
\end{lemma}

\begin{proof}
The morphism of locally free sheaves
$$T\colon \beta^*T_X(\textup{log}\,D) \cong \beta^*\big(\beta_*T_Y(\textup{log}\,E)\big) \to T_Y(\textup{log}\,E)$$
yields a map
$$\det T\colon \beta^*\sO_{X}(-K_X-D) \cong \beta^* \det T_X(\textup{log}\,D) \to \det T_Y(\textup{log}\,E)\cong \sO_Y(-K_Y-E),$$
which is an isomorphism over $Y \setminus \beta(\textup{Exc}\,\beta)$. Since $(X,D)$ is log canonical, $\det T$ must be an isomorphism. This immediately implies that $T$ is an isomorphism as well, proving the lemma.
\end{proof}

\subsection{R-flat vector bundles}\label{R_flat_vector_bundles} For the reader's convenience, we recall the notion of \textit{numerical flatness} for vector
bundles.

\begin{defn}
\label{defn: numerically flat}
Let $f \colon X \to Y$ be a projective morphism of quasi-projective varieties. A locally free sheaf $\sE$ on $X$ of positive rank is called
$f$-numerically flat if $\sE$ and $\sE^*$ are $f$-nef. If $Y$ is a point, we simply say that $\sE$ is numerically flat.
\end{defn}

In the course of the proof of \cite[Theorem 1.1]{biswas_vb_rc} the authors show that a locally free sheaf $\sE$ of positive rank $r$ on a smooth projective rationally connected variety $X$ such that $\nu^*\sE \cong \sO_{\mathbb{P}^1}^{\oplus r}$ for any morphism $\nu\colon\mathbb{P}^1 \to X$
is numerically flat. Definition \ref{def:R_flat} is a formalization of the above condition.

\begin{defn}\label{def:R_flat}
Let $f \colon X \to Y$ be a projective morphism of quasi-projective varieties. A locally free sheaf $\sE$ on $X$ of positive rank $r$ is called relatively R-flat or $f$-relatively R-flat if $\nu^*\sE \cong \sO_{\mathbb{P}^1}^{\oplus r}$ for any morphism $\nu\colon\mathbb{P}^1 \to X$ such that $\nu(\mathbb{P}^1)$ is contracted by $f$. If $Y$ is a point, we simply say that $\sE$ is R-flat. 
\end{defn}

The following result partly extends \cite[Theorem 1.1]{biswas_vb_rc} to our setting.

\begin{lemma}\label{lemma:R_flat_numerically_flat}
Let $X$ be a projective reduced space, not necessarily irreducible. Suppose that $X$ is rationally chain connected.
Then any locally free, R-flat sheaf on $X$ is numerically flat.
\end{lemma}

\begin{proof} Let $\sE$ be a locally free, R-flat sheaf on $X$. Let $C$ be a curve on $X$ and let
$x \in X$ be a general point. Since $X$ is rationally chain connected,
there exist finitely many normal projective surfaces $S_i$ ($1\le i \le N$) as well as surjective morphisms 
$\pi_i \colon S_i \to B_i$ onto smooth complete curves such that the following holds. The general fibers of $\pi_i$ are rational curves. Moreover, there exist sections $\sigma_{i,1}\subset S_i$ and $\sigma_{i,2}\subset S_i$ of $\pi_i$ and morphisms $e_i \colon S_i \to X$ such that 
$e_1(\sigma_{1,1})=\{x\}$, $e_i(\sigma_{i,2})=e_{i+1}(\sigma_{i+1,1})$ for $1\le i \le N-1$
and $e_N(\sigma_{N,2})=C$. By the semistable reduction theorem, we may assume without loss of generality that $\pi_i$ is semistable.
In particular, any fiber of $\pi_i$ is a rational tree. Notice also that $\pi_i$ is flat. Since $\sE$ is R-flat, we must have $e_i^*\sE \cong \pi_i^*\sG_i$ for some locally free sheaf $\sG_i$ on $B_i$. Now, ${e_1^*\sE}_{|\sigma_{1,1}}$ is trivial since $e_1(\sigma_{1,1})=\{x\}$. It follows that 
$\sG_1$ is numerically flat and hence so is $e_1^*\sE$. This in turn implies that ${\sE}_{|e_1(\sigma_{1,2})}=\sE_{|e_2(\sigma_{2,1})}$ is numerically flat as well if $\dim e_1(\sigma_{1,2}) = 1$. If $\dim e_1(\sigma_{1,2}) = 0$, then ${\sE}_{|e_1(\sigma_{1,2})}=\sE_{|e_2(\sigma_{2,1})}$ is the trivial vector bundle.
An induction on $i$ then shows that $\sE_{|C}$ is numerically flat, finishing the proof of the lemma.
\end{proof}

The following lemma will be very useful.

\begin{lemma}\label{lemma:R_flat_functoriality}
Let $f \colon X \to Y$ be a projective morphism of quasi-projective varieties, and let $\sE$ be a locally free sheaf on $Y$. Suppose that $X$ is of klt type and that $-K_X$ is $f$-ample. If $f^*\sE$ is R-flat, then so is $\sE$.
\end{lemma}

\begin{proof}
This is an immediate consequence of \cite[Corollary 1.10]{hacon_mckernan}.
\end{proof}

\subsection{Discriminant}
Let $f \colon X \to Y$ be a projective morphism with connected fibers of normal, quasi-projective varieties. Let $D$ be a $\mathbb{Q}$-divisor on $X$ such that $(X,D)$ is log canonical over the generic point of $Y$. The \textit{discriminant divisor} of $(f,D)$
is the $\mathbb{Q}$-divisor $B=\sum_P b_P P$ on $Y$, where $P$ runs through all prime divisor on $Y$, and 
$$1-b_P:=\sup\left\{ t\in \mathbb{R}\,| \,(X,D+tf^*P)\textup{ is log canonical over the generic point of }P\right\}.$$   
The discriminant divisor measures the singularities of special fibers. For the reader's convenience, we recall three standard facts.

\begin{fact}
If $C$ is a $\mathbb{Q}$-divisor on $Y$, then the discriminant divisor of $(f,D+f^*C)$ is $B+C$.
\end{fact}

\begin{fact}[{\cite[Lemma 2.6]{ambro_jdg}}]\label{fact:discriminant_crepant} 
Suppose that $K_X+D$ is $\mathbb{Q}$-Cartier.
Let $\beta\colon Z \to X$ be a projective birational morphism with $Z$ normal. Set $g:=f\circ\beta$ and
$C:=\beta^*(K_X+D)-K_Z$. Then the discriminant divisor of $(g,C)$ is $B$.
\end{fact}

\begin{fact}[{\cite[Lemma 5.1]{ambro_jdg}}]\label{fact:discriminant_finite_base_change} 
Let $\gamma\colon Y_1 \to Y$ be a finite morphism with $Y_1$ normal, and let $X_1$ be the normalization of the product $Y_1 \times_X X_1$ with natural morphisms $f_1\colon X_1 \to Y_1$ and $\gamma_1 \colon X_1 \to X$. Set 
$D_1:=\gamma_1^*(K_X+D)-K_{X_1}$. Let $B_1$ be the discriminant divisor of $(f_1,D_1)$. Then $K_{X_1}+B_1\sim_\mathbb{Z}\gamma_1^*(K_X+B)$.
\end{fact}

We will also need the following observation.

\begin{lemma}\label{lemma:cbf_2}
Let $X$ be a normal projective variety, and let $D$ be an effective $\mathbb{Q}$-divisor on $X$. Suppose that $(X,D)$ has klt singularities. Let $\psi \colon X \to Y$ be a projective morphism with connected fibers onto a normal, projective variety $Y$, and let $B$ the discriminant divisor of $(\psi,D)$. Suppose that $-(K_X+D)$ is $\psi$-ample.
If $K_Y+B$ is $\mathbb{Q}$-Cartier, then $(Y,B)$ has klt singularities.
\end{lemma}

\begin{proof}Let $H_Y$ be an ample $\mathbb{Q}$-divisor on $Y$ such that $H_X:=-(K_X+D)+\psi^*H_Y$ is ample, and let $m \ge 2$ be an integer such that $mH_X$ is very ample. Let $D_1\in|mH_X|$ be a general member. By general choice of $D_1$, we may assume that $(X,D+\frac{1}{m}D_1)$ has klt singularities and that the discriminant divisor of $(\psi,D+\frac{1}{m}D_1)$ is $B$. Notice that $K_X+D+\frac{1}{m}D_1\sim_\mathbb{Q} \psi^*H_Y$.
Arguing as in the proof of \cite[Theorem 1.2]{fujino99}, one then shows that there exists an effective divisor $B_1$ on $Y$ such that $(Y,B_1)$ is klt and $B \le B_1$. This immediately implies that $(Y,B)$ is klt, proving the lemma. 
\end{proof}

\section{A canonical bundle formula for generically isotrivial fibrations}

In this section we prove a canonical bundle formula for generically isotrivial fibrations (see Theorem \ref{thm:cbf}), which might be of independent interest. The canonical bundle formula for the so-called lc-trivial fibrations is a higher dimensional analogue of Kodaira's canonical bundle formula for minimal elliptic surfaces. We refer to \cite{kollar_cbf} and the references therein for a (rather delicate) precise formulation. 

We will use the following definition.

\begin{defn}\label{defn:isotrivial}
Let $f \colon X \to Y$ be a morphism with connected fibers of algebraic varieties, and let $D$ be an integral divisor on $X$. Suppose that $D$ is effective and reduced over the generic point of $Y$. The fibration $(f,D)$ is called \textit{generically isotrivial} if there is a dense Zariski open set $Y^\circ \subseteq Y$ such that the following holds. For every point $y \in Y^\circ$, there exists a Euclidean open neighbourhood $U$ of $y$ in $Y^\circ$ such that $$\big(f^{-1}(U),D_{|f^{-1}(U)}\big)\cong \big(U \times F,U \times D_{|F}\big)$$ over $U$, where $F$ if a fiber of $f_{|f^{-1}(U)}$.  
\end{defn}

\begin{rem}\label{remark:isotriviality}
Setup as in Definition \ref{defn:isotrivial}. Suppose in addition that $f$ is projective.   
Then $(f,D)$ is generically isotrivial if and only if there exists a dense Zariski open set $Y^\circ \subseteq Y$ and a finite morphism $\gamma^\circ\colon Y_1^\circ \to Y^\circ$ such that $$\big(X_1^\circ,{D_1^\circ}_{|X_1^\circ}\big)\cong \big(Y_1^\circ \times F_1,Y_1^\circ \times {D_1^\circ}_{|F_1}\big)$$ over $Y_1^\circ$, where $X_1^\circ$ denotes the normalization of $Y_1^\circ \times_{Y^\circ} X$ with natural morphisms $f_1^\circ \colon X_1^\circ \to Y_1^\circ$ and $\gamma_1^\circ\colon X_1^\circ \to X^\circ=:f^{-1}(Y^\circ)$, $F_1$ is a general fiber of $f_1^\circ$, and $D_1^\circ:=(\gamma_1^\circ)^*(D_{|X^\circ})$. 
Indeed, shrinking $Y$, if necessary, we may assume that $f$ is flat and that the restriction of $f$ to any irreducible component of $D$ is also flat over $Y$. The claim is then an easy consequence of the fact that the space $\textup{Isom}_{Y}\big((X,D),(Y\times F,Y\times D_{|F})\big)$
parametrizing isomorphims of pairs $(X_y,D_{|X_y}) \cong (F,D_{|F})$ with $y\in Y$ is quasi-projective over $Y$.  
\end{rem}

\begin{thm}\label{thm:cbf}
Let $X$ be a normal projective variety, and let $f \colon  X \to Y$ be a surjective morphism with connected fibers onto a normal projective variety $Y$.
Let also $D$ be a Weil divisor on $X$. Suppose that $D$ is effective in a neighbourhood of a general fiber of $f$ and that 
$(X,D)$ is log canonical over the generic point of $Y$.
Suppose in addition that there exists a Cartier divisor $C$ on $Y$ such that $K_X+D\sim_\mathbb{Q} f^* C$. 
If $(f,D)$ is generically isotrivial, then $C \sim_\mathbb{Q} K_Y + B$, where $B$ denotes the discriminant divisor of $(f,D)$.
\end{thm}

\begin{proof}
Let $m$ be smallest positive integer such that $m(K_X+D)\sim_\mathbb{Z} mf^* C$, and let $\gamma \colon X_1 \to X$ be the corresponding index one cover, which is quasi-\'etale (see \cite[Definition 2.52]{kollar_mori}). Set $D_1:=\gamma^*(K_X+D)-K_{X_1}$.
By choice of $m$, the morphism $f_1:=f\circ \gamma$ has connected fibers.
Notice that $D_1$ is effective in a neighbourhood of a general fiber of $f_1$ and that the pair $(X_1,D_1)$ is log canonical over the generic point of $Y$. Moreover, by construction, $K_{X_1}+D_1\sim_\mathbb{Z}f_1^*C$. By Fact \ref{fact:discriminant_crepant}, $B$ is also the discriminant divisor of $(f_1,D_1)$. Finally, the fibration $(f_1,D_1)$ is easily seen to be generically isotrivial as well 
using Lemma \ref{lemma:fundamental_group} below.
Replacing $(f,D)$ by $(f_1,D_1)$, if necessary, we may therefore assume without loss of generality that the following holds.

\begin{assumption}
The relation $K_X+D\sim_\mathbb{Z} f^* C$ holds. In particular, $K_X+D$ is Cartier.
\end{assumption}

By assumption (see also Remark \ref{remark:isotriviality}), there exists a finite cover $\gamma \colon Y_1 \to Y$ such that the following holds. Let $X_1$ be the normalization of $Y_1 \times_Y X$ with natural morphisms $f_1 \colon X_1 \to Y_1$ and $\tau\colon X_1 \to X$. 
Set $D_1:=\tau^*(K_X+D)-K_{X_1}$ and let $F_1$ be a general fiber of $f_1$. Then there exist a dense open set $Y_1^\circ \subseteq Y_1$ and an isomorphism of pairs
$$\big(f_1^{-1}(Y_1^\circ),{D_1}_{|f^{-1}(Y_1^\circ)} \big)\cong  \big(Y_1^\circ \times F_1,Y_1^\circ \times {D_1}_{|F_1}\big)$$ over 
$Y_1^\circ$.  
Let $B_1$ denotes the discriminant divisor of $(f_1,D_1)$.
Let $\gamma_1\colon Y_2 \to Y_1$ be a resolution of singularities, and let $\tau_1 \colon X_2 \to X_1$
be a resolution of the main component of the product $Y_2 \times_{Y_1} X_1$ with natural morphism $f_2\colon X_2 \to Y_2$. 
Set $D_2:=\tau_1^*(K_{X_1}+D_1)-K_{X_2}$.
Set also $Y_2^\circ := \gamma_1^{-1}(Y_1^\circ)$ and $X_2^\circ:=f_2^{-1}(Y_2^\circ)$.
We may assume without loss of generality that $Y_2 \setminus Y_2^\circ$ has
simple normal crossings and that $$\big(X_2^\circ,{D_2}_{|X_2^\circ}\big) \cong 
\big(Y_2^\circ \times F_2,Y_2^\circ \times {D_2}_{|F_2}\big)$$
over $Y_2^\circ$, where $F_2$ denotes a general fiber of $f_2$.
In particular, ${f_2}_{|X_2^\circ}\colon X_2^\circ \to Y_2^\circ$ is a smooth morphism.
Applying \cite[Theorem 6.4]{viehweg_moduli_polarized_manifolds}, we see that there exists a finite morphism
$\gamma_2\colon Y_3 \to Y_2$ of complex manifolds and a resolution $X_4$
of the normalization $X_3$ of the product $Y_3\times_{Y_2} X_2$ such that the following holds.
Let $f_4\colon X_4 \to Y_3$ denote the natural morphism, and set $Y_3^\circ := \gamma_2^{-1}(Y_2^\circ)$. Then 
$Y_3\setminus Y_3^\circ$ and $f_4^{-1}\big(Y_3\setminus Y_3^\circ\big)$
have simple normal crossings, and $f_4$ has reduced fibers in codimension one.
We obtain a commutative diagram as follows:
\begin{center}
\begin{tikzcd}[row sep=large]
X_4\ar[d, "{f_4}"']\ar[rrr, "{\tau_3, \textup{ birational}}"] &&& X_3\ar[d, "{f_3}"]\ar[rr, "{\tau_2,\textup{ finite}}"] && X_2 \ar[d, "{f_2}"]\ar[rrr, "{\tau_1, \textup{ birational}}"] &&& X_1  \ar[d, "{f_1}"]\ar[rr, "{\tau,\textup{ finite}}"] & & X\ar[d, "{f}"] \\
Y_3 \ar[rrr, equal] &&& Y_3 \ar[rr, "{\gamma_2,\textup{ finite}}"'] && Y_2 \ar[rrr, "{\gamma_1, \textup{ birational}}"'] &&& Y_1 \ar[rr, "{\gamma, \textup{ finite}}"']&& Y.
\end{tikzcd}
\end{center}
Set $D_3:=\tau_2^*(K_{X_2}+D_2)-K_{X_3}$ and $D_4:=\tau_3^*(K_{X_3}+D_3)-K_{X_4}$.
Notice that $D_4$ has integral coefficients since $K_X+D$ is Cartier by assumption.
Set also $X_4^\circ:=f_4^{-1}(Y_3^\circ)$.
Finally,
we may assume without loss of generality that 
$X_4$ is a log resolution of $(X_3,D_3)$ and that 
$$\big(X_4^\circ,{D_4}_{|X_4^\circ}\big) \cong 
\big(Y_3^\circ \times F_4,Y_3^\circ \times {D_4}_{|F_4}\big)$$
over $Y_3^\circ$, where $F_4$ denotes a general fiber of $f_4$.

Write $D_4=R_4+S_4-E_4-G_4$, where $R_4$, $S_4$, $E_4$ and $G_4$ are effective divisors with no common components such that any irreducible component of $R_4+E_4$ maps onto $Y_3$ and any irreducible component of $S_4+G_4$ maps into a proper subset of $Y_3$. 
Observe that the divisor $R_4$ is reduced since $(X,D)$ is log canonical over the generic point of $Y$. Moreover, since $D$ is effective in a neighbourhood of a general fiber of $f$, any irreducible component of 
$E_4$ is exceptional over $X_1$. 
Blowing-up strata of $R_4$ mapping into a proper subset of $Y_3$, if necessary, we may also assume that any stratum of $R_4$ dominates $Y_3$. 
Set $f_4^\circ:={f_4}_{|X_4^\circ\setminus \textup{Supp}(R_4)}$ and $d:=\dim X_4 - \dim Y_3$. By construction, the local system
$(R^d f_4^\circ)_*\mathbb{C}_{X_4^\circ\setminus \textup{Supp}(R_4)}$ is trivial. In particular, the Deligne's canonical extension of 
$(R^d f_4^\circ)_*\mathbb{C}_{X_4^\circ\setminus \textup{Supp}(R_4)}$ is the trivial local system on $Y_3$ (see \cite[pp. 91-95]{deligne_equ_diff}).
Now, by \cite[Theorems 3.1 and 3.9]{fujino_higher}, the sheaf $(f_4)_*\sO(K_{X_4/Y_3}+R_4)$ is locally free (of rank $1$) and numerically effective. Moreover, it identifies with the so-called (upper) canonical extension of 
$$F^d \big((R^d f_4^\circ)_*\mathbb{C}_{X_4^\circ\setminus \textup{Supp}(R_4)}\otimes \sO_{Y_3^\circ}\big)
\subset (R^d f_4^\circ)_*\mathbb{C}_{X_4^\circ\setminus \textup{Supp}(R_4)}\otimes \sO_{Y_3^\circ},$$ where $F^{\bullet}\big((R^d f_4^\circ)_*\mathbb{C}_{X_4^\circ\setminus \textup{Supp}(R_4)}\otimes \sO_{Y_3^\circ}\big)$ denotes the Hodge filtration (see \cite[Section 3.1]{fujino_higher} and the references therein). This immediately implies that 
$$(f_4)_*\sO(K_{X_4/Y_3}+R_4)\cong \sO_{Y_3},$$
and hence
$$\sO_{Y_3}(C_3) \otimes (f_4)_*\sO_{X_4}(E_4+G_4-S_4) \cong  \sO_{Y_3}(K_{Y_3}),$$
where $C_3$ denotes the pull-back of $C$ to $Y_3$. 

Let $B_3^{+}$ be the smallest effective divisor on $Y_3$ such that 
$S_4 \le f_4^*B_3^{+}$ over the codimension $1$ points of $f_4\big(\textup{Supp}\,S_4\big)$
and let $B_3^{-}$ be the largest (effective) divisor on $Y_3$ such that 
$\textup{Supp}\, B_3^{-} \subseteq f_4(\textup{Supp}\,G_4)$ and 
$f_4^*B_3^{-} \le G_4$
over the codimension $1$ points of $f_4(\textup{Supp}\,G_4)$.
By construction, there is an open set $W \subseteq Y_3$ with complement of codimension at least $2$ and a natural inclusion
$${\sO_{Y_3}(B_3^{-}-B_3^{+})}_{|W} \subseteq {(f_4)_*\sO_{X_4}(E_4+G_4-S_4)}_{|W}.$$ 
We now show that this map is an isomorphism.
Let $U \subseteq W$ be dense open set, and let $t$ be a rational function on $X_4$ such that 
$$(\div\,t)_{|f_4^{-1}(U)}+(E_4+G_4-S_4)_{|f_4^{-1}(U)} \ge 0.$$
Since $E_4$ is exceptional over $X_1$, any regular function on $F_4 \setminus \textup{Supp}\,{E_4}_{|F_4}$ is constant.
This immediately implies that 
$$(\div\,t)_{|f_4^{-1}(U)}+(G_4-S_4)_{|f_4^{-1}(U)} \ge 0.$$
Moreover, since $$\big(X_4^\circ,{D_4}_{|X_4^\circ}\big) \cong 
\big(Y_3^\circ \times F_4,Y_3^\circ \times {D_4}_{|F_4}\big),$$ there exists a rational function $r$ on $Y$
such that $t=r\circ f_4$. 
One then easily checks that $$(\div\,r)_U+{(B_3^{-}-B_3^{+})}_{|U} \ge 0$$
using the fact that $f_4$ has reduced fibers in codimension $1$.
This shows that $${\sO_{Y_4}(B_3^{-}-B_3^{+})}_{|W} \cong {(f_4)_*\sO_{X_4}(E_4+G_4-S_4)}_{|W},$$ and hence 
$$(f_4)_*\sO_{X_4}(E_4+G_4-S_4) \cong  \sO_{Y_3}(B_3^{-}-B_3^{+})$$
since both sheaves are locally free on $Y_3$.

Next, we show that the discriminant divisor $B_4$ of $(f_4,D_4)$
is $B_3^{+} - B_3^{-}.$
Let $P$ be a prime divisor on $Y$. If $P$ is not contained in $f_4(\textup{Supp}\,G_4)\cup f_4(\textup{Supp}\,S_4)$, then $P$ is obviously not contained in the supports of $B_3^{+}$ and $B_3^{-}$. Moreover, $P$ is not contained in the support of $B_4$ since $R_4$ is a relative normal crossings divisor over $Y_3 \setminus Y_3^\circ$ by construction.  
Suppose that $P$ is contained in $f_4(\textup{Supp}\,G_4)\cup f_4(\textup{Supp}\,S_4)$. 
Then $\mult_P(B_3^{-})=0$ since $S_4$ and $G_4$ have no common components. 
One then readily checks that $\mult_P (B_4)=\mult_P(B_3^{+} - B_3^{-})$ using the fact that $f_4$ has reduced fibers in codimension $1$.
This shows that the discriminant divisor of $(f_4,D_4)$ is $B_3^{+} - B_3^{-}$ and thus,
$$C_3 \sim_\mathbb{Z} K_{Y_3}+B_4.$$
By Fact \ref{fact:discriminant_crepant}, we have $B_4=B_3$. On the other hand, by Fact \ref{fact:discriminant_finite_base_change},
$K_{Y_1}+B_1 \sim_\mathbb{Q}\gamma^*(K_Y+B)$ and $K_{Y_3}+B_3 \sim_\mathbb{Q}\gamma_2^*(K_{Y_2}+B_2)$.
Finally, $(\gamma_1)_*B_2$ is obviously the discriminant divisor of $(f_1\circ \tau_1,D_2)$, and hence 
$(\gamma_1)_*B_2=B_{Y_1}$ by Fact \ref{fact:discriminant_crepant} again.
This shows that 
$$C \sim_\mathbb{Q} K_Y + B,$$
completing the proof of the theorem.
\end{proof}

\begin{lemma}\label{lemma:fundamental_group}
Let $X$ be a normal variety and let $d$ be a positive integer. Then there are only finitely many Galois quasi-\'etale covers of $X$ of degree $d$ up to isomorphism.
\end{lemma}

\begin{proof}
By the Nagata-Zariski purity theorem, any quasi-\'etale cover of $X$ branches only on the singular set of $X$. Thus, by the Riemann existence theorem (see \cite[Expos\'e XII, Th\'eor\`eme 5.1]{sga1}), we need to show that there are only finitely many normal subgroups of the topological fundamental group $\pi_1(X_{\textup{reg}})$ of index $d$. But this follows easily from the fact that the group $\pi_1(X_{\textup{reg}})$ is finitely generated, proving the lemma.  
\end{proof}

The following is an immediate consequence of Theorem \ref{thm:cbf}.

\begin{cor}\label{cor:cbf}
Let $X$ be a normal projective variety, and let $f \colon  X \to Y$ be a surjective morphism with connected fibers onto a normal projective variety $Y$.
Let $D$ be a reduced effective divisor on $X$ such that $(X,D)$ is log canonical with $K_X+D\sim_\mathbb{Q}0$.
Suppose furthermore that $(f,D)$ is generically isotrivial. Then $K_Y + B \sim_\mathbb{Q} 0$, where $B$ denotes the discriminant divisor of $(f,D)$.
\end{cor}

\section{Descent of vector bundles}

The proof of Theorem \ref{thm:main_descent_intro} relies on the following auxiliary statement. To put the result into perspective, consider a Mori extremal contraction $f \colon X \to Y$ of a projective klt space. By the cone theorem, a line bundle on $X$ of degree zero on every contracted rational curve comes from $Y$. We first generalize this result to vector bundles of arbitrary rank. The special case where $\dim X = \dim Y$ follows from \cite[Theorem 4.1]{kebekus_al_naht} together with Lemma \ref{lemma:R_flat_numerically_flat}.

\begin{thm}\label{thm:descent}
Let $f \colon X \to Y$ be a projective morphism with connected fibers of normal, quasi-projective varieties. Suppose that 
there is an effective $\mathbb{Q}$-divisor $D$ on $X$ such that the pair $(X,D)$ is klt and $-(K_X+D)$ is $f$-ample. Let $\sE$ be a locally free, $f$-relatively R-flat sheaf on $X$. Then there exists a locally free sheaf $\sG$ on $Y$ such that
$\sE \cong f^*\sG$.
\end{thm}

\begin{proof}
By \cite[Theorem 1.2]{hacon_mckernan}, every fiber of $f$ is rationally chain connected. Together with Lemma \ref{lemma:R_flat_numerically_flat}, this implies that $\sE$ is $f$-numerically flat.

If $\dim X = \dim Y$, then Theorem \ref{thm:descent} follows from \cite[Theorem 4.1]{kebekus_al_naht}. Suppose from now on that 
$\dim Y < \dim X$. The proof is similar to that of \textit{loc. cit.} and so we leave some easy details for the reader. 

Let $\sO_{\mathbb{P}_X(\sE)}(1)$ denote the tautological line bundle on $\mathbb{P}_X(\sE)$, and let $p\colon \mathbb{P}_X(\sE) \to X$ denote the projection map. 
Let $C$ be a divisor on $X$ such that $\sO_{\mathbb{P}_X(\sE)}(C)\cong\sO_{\mathbb{P}_X(\sE)}(1)$ and set $r:=\textup{rank}\,\sE$. Note that the pair $\big(\mathbb{P}_X(\sE),p^*D\big)$ is klt. 

Set $\psi:=f\circ p$, and let $F$ be a general fiber of $f$. Note that $F$ is of klt type and rationally chain connected.
By \cite[Theorem 1.1]{biswas_vb_rc} applied to a resolution of $F$ together with \cite[Theorem 1.2]{hacon_mckernan}, we must have $\sE_{|F}\cong \sO_F^{\oplus r}$.
Moreover, by the adjunction formula, $-(K_F+D_{|F})$ is ample.
This implies that the restriction of the $\mathbb{Q}$-divisor
$$C-\big(K_{\mathbb{P}_X(\sE)}+p^*D\big)\sim_\mathbb{Z}(r+1)\cdot C-p^*\big(K_X+D+c_1(\sE)\big)$$ to $G:=p^{-1}(F) \cong \mathbb{P}^{r-1} \times F$ is ample.
On the other hand, $C$ is $\psi$-nef since $\sE$ is $f$-numerically flat. By the base-point-free theorem (see \cite[Theorem 3.1.1 and Remark 3.1.2]{kmm}), we conclude that there is a factorization of $\psi$ via a normal variety $Z$
\begin{center}
\begin{tikzcd}[row sep=large, column sep=large]
\mathbb{P}_X(\sE) \ar[r, two heads, "{g}"] \ar[d, "{p}"'] & Z \ar[d, two heads, "{q}"]\\
X  \ar[r, "{f}"'] & Y
\end{tikzcd}
\end{center}
such that $g$ has connected fibers and such that $\sO_{\mathbb{P}_X(\sE)}(1)$ is the pull-back of a $q$-ample line bundle $\sL$ on $Z$.
By construction, the restriction of $\sO_{\mathbb{P}_X(\sE)}(1)$ to $G=\psi^{-1}(F) \cong \mathbb{P}^{r-1} \times F$ is isomorphic to the pull-back of $\sO_{\mathbb{P}^{r-1}}(1)$ to $\mathbb{P}^{r-1} \times F$ via the projection $\mathbb{P}^{r-1} \times F \to \mathbb{P}^{r-1}$. This easily implies that 
$$\dim Z = \dim Y + r-1.$$
Moreover, a general fiber of $q$ is isomorphic to $\mathbb{P}^{r-1}$ and the restriction of $\sL$ to this fiber is isomorphic to 
$\sO_{\mathbb{P}^{r-1}}(1)$.

Next, we show that $q$ is equidimensional of relative dimension $r-1$. We argue by contradiction and assume that there exists a variety $T \subseteq Z$ with $\dim T = r$ and 
$\dim q(T)=0$. Let $F_1$ be an irreducible component of $f^{-1}\big(q(T)\big)$ such that the restriction of $g$ to 
$p^{-1}(F_1)$ induces a surjective morphism $p^{-1}(F_1) \twoheadrightarrow T$. We obtain a diagram as follows:

\begin{center}
\begin{tikzcd}[row sep=large, column sep=large]
p^{-1}(F_1)=\mathbb{P}_{F_1}(\sE_{|F_1}) \ar[r, two heads, "{g_{p^{-1}(F_1)}}"] \ar[d, "{p_{|p^{-1}(F_1)}}"'] & T \\
F_1.  & 
\end{tikzcd}
\end{center}
Since $\sL$ is $q$-ample, we have $(\sL_{|T})^r \neq 0$. Now, by construction, the pull-back of $\sL_{|T}$ to $p^{-1}(F_1)$ identifies with 
the tautological line bundle $\sO_{\mathbb{P}_{F_1}(\sE_{|F_1})}(1)$. On the other hand, $\sE_{|F_1}$ is numerically flat. 
Thus, by \cite[Corollary 1.19]{demailly_peternell_schneider94} applied to the pull-back of $\sE_{|F_1}$ to a resolution of $F_1$ together with the projection formula, we must have $c_1(\sE_{|F_1})\equiv 0$ and  $c_2(\sE_{|F_1})\equiv 0$.
But then \cite[Remark 3.2.4]{fulton} gives $\sO_{\mathbb{P}_{F_1}(\sE_{|F_1})}(1)^r \equiv 0$, yielding a contradiction. This shows that $q$ is equidimensional.
By \cite[Proposition 4.10]{codim_1_del_pezzo_fols}, there is a vector bundle $\sG$ on $Y$ such that 
$$(Z,\sL)\cong \big(\mathbb{P}_Y(\sG),\sO_{\mathbb{P}_Y(\sG)}(1)\big).$$
But then, we must have $\mathbb{P}_X(\sE) \cong X\times_Y \mathbb{P}_Y(\sG)$. This immediately implies that 
$\sE \cong f^*\sG$,
completing the proof of the theorem.
\end{proof}

\begin{proof}[Proof of Theorem \ref{thm:main_descent_intro}] Notice that the locally free sheaf $\sE$ is $f$-numerically flat by Lemma \ref{lemma:R_flat_numerically_flat}.

We prove Theorem \ref{thm:main_descent_intro} by induction on $\dim X - \dim Y$.

If $\dim X = \dim Y$, then Theorem \ref{thm:main_descent_intro} follows from \cite[Theorem 5.1]{kebekus_al_naht}. Suppose from now on that 
$\dim Y < \dim X$.

Then $X$ is uniruled, and thus, we may run a minimal model program for $X$ and end with a Mori fiber space (see \cite[Corollary 1.3.3]{bchm}). There exists a sequence of maps

\begin{center}
\begin{tikzcd}[row sep=large, column sep=large]
X:=X_1 \ar[r, "{\phi_1}", dashrightarrow]\ar[d, "{f_1:=f}"] & \cdots \ar[r, "{\phi_{i-1}}", dashrightarrow] & X_i \ar[r, "{\phi_i}", dashrightarrow]\ar[d, "{f_i}"] & X_{i+1} \ar[r, "{\phi_{i+1}}", dashrightarrow]\ar[d, "{f_{i+1}}"] & \cdots \ar[r, "{\phi_{m-1}}", dashrightarrow] & X_{m} \ar[r, "{\psi}"]\ar[d, "{f_m}"] & X_{m+1}\ar[d, "{f_{m+1}}"]\\
Y \ar[r, equal] & \cdots\ar[r, equal] & Y\ar[r, equal] & Y\ar[r, equal] & \cdots\ar[r, equal] & Y \ar[r, equal] & Y
\end{tikzcd}
\end{center}

\noindent where the $\phi_i$ are either divisorial contractions or flips, and $\psi$ is a Mori fiber space. The spaces $X_i$ are normal, $\mathbb{Q}$-factorial, and $X_i$ has klt singularities for all $1\le i \le m$. Moreover, by \cite[Lemma 5.1.5]{kmm}, $X_{m+1}$ is also $\mathbb{Q}$-factorial. Applying \cite[Corollary 4.6]{fujino99}, we see that $X_{m+1}$ is klt as well. 

We construct smooth projective varieties $Z_i$ inductively for any integer $1 \le i \le m+1$ as follows. Let $Z_{m+1} \to X_{m+1}$ be a resolution of $X_{m+1}$, and let 
$Z_i$ be a resolution of the graph of the rational map $X_i \map Z_{i+1}$ for $1 \le i \le m$. We obtain a commutative diagram as follows:

\begin{center}
\begin{tikzcd}[row sep=large, column sep=large]
Z_1 \ar[r]\ar[d, "{\beta_1}"]\ar[rr, bend left, "{g_i}"']\ar[rrr, bend left, "{g_{i+1}}"']\ar[rrrrr, bend left, "{g_m}"']\ar[rrrrrr, bend left, "{g_{m+1}}"'] & \cdots\ar[r] & Z_i\ar[r]\ar[d, "{\beta_i}"] & Z_{i+1}\ar[r]\ar[d, "{\beta_{i+1}}"] & \cdots\ar[r] & Z_m \ar[r]\ar[d, "{\beta_m}"] & Z_{m+1}\ar[d, "{\beta_{m+1}}"]\\  
X:=X_1 \ar[r, "{\phi_1}", dashrightarrow]\ar[d, "{f_1:=f}"] & \cdots \ar[r, "{\phi_{i-1}}", dashrightarrow] & X_i \ar[r, "{\phi_i}", dashrightarrow]\ar[d, "{f_i}"] & X_{i+1} \ar[r, "{\phi_{i+1}}", dashrightarrow]\ar[d, "{f_{i+1}}"] & \cdots \ar[r, "{\phi_{m-1}}", dashrightarrow] & X_{m} \ar[r, "{\psi}"]\ar[d, "{f_m}"] & X_{m+1}\ar[d, "{f_{m+1}}"]\\
Y \ar[r, equal] & \cdots\ar[r, equal] & Y\ar[r, equal] & Y\ar[r, equal] & \cdots\ar[r, equal] & Y \ar[r, equal] & Y
\end{tikzcd}
\end{center}

Next, we show inductively that there exist locally free and $f_i$-numerically flat sheaves $\sE_i$ on $X_i$ such that 
$\beta_1^*\sE \cong (\beta_i \circ g_i)^*\sE_i$. Set $\sE_1:=\sE$.

Let $1\le i \le m-1$. Suppose first that $\phi_i$ is a divisorial contraction. Then Theorem \ref{thm:descent} shows that there exists a locally free sheaf $\sE_{i+1}$ on $X_{i+1}$ such that $\sE_i \cong \phi_i^*\sE_{i+1}$. Note that $\sE_{i+1}$ is obviously $f_{i+1}$-numerically flat and that  
$\beta_1^*\sE \cong (\beta_{i+1} \circ g_{i+1})^*\sE_{i+1}$.

Suppose now that $\phi_i$ is the flip of a small extremal contraction $c_i\colon X_i \to Y_i$ over $Y$, and let $c_{i+1}\colon X_{i+1} \to Y_i$ be the natural $Y$-morphism. By Theorem \ref{thm:descent}, there exists a locally free sheaf $\sG_i$ on $Y_i$ such that 
$\sE_i\cong c_i^*\sG_i$. Set $\sE_{i+1}:=c_{i+1}^*\sG_i$. Then $\sE_{i+1}$ is $f_{i+1}$-numerically flat and 
$\beta_1^*\sE \cong (\beta_{i+1} \circ g_{i+1})^*\sE_{i+1}$. 

If $i=m$, then Theorem \ref{thm:descent} again shows that there exists a locally free sheaf $\sE_{m+1}$ on $X_{m+1}$ such that $\sE_m \cong \psi^*\sE_{m+1}$. The sheaf $\sE_{m+1}$ is $f_{m+1}$-numerically flat and $\beta_1^*\sE \cong (\beta_{m+1} \circ g_{m+1})^*\sE_{m+1}$.

Now, a general fiber of $f$ is rationally chain connected by assumption and hence rationally connected since it is smooth. This easily implies that the general fibers of $f_{m+1}\circ\beta_{m+1}\colon Z_{m+1} \to Y$ are rationally connected.
The induction hypothesis applied to $\beta_{m+1}^*\sE_{m+1}$ then implies that there exists a locally free sheaf $\sG$ on $Y$ such that $\beta_{m+1}^*\sE_{m+1}\cong (f_{m+1}\circ\beta_{m+1})^*\sG$. One readily checks that $\sE \cong f^*\sG$, completing the proof of the corollary.  
\end{proof}

\section{Preparation for the proof of the main results}

In the section we provide a technical tool for the proof of Theorem \ref{thm:main_intro}. We will prove Theorem \ref{thm:main_intro} by induction on the dimension. The following result will be useful for the induction process.

\begin{prop}\label{prop:main}
Let $X$ be a normal projective variety, and let $D$ be a reduced effective divisor on $X$ such that $(X,D)$ is log canonical and $K_X+D\sim_\mathbb{Q}0$. Suppose that $X$ is $\mathbb{Q}$-factorial with klt singularities. Suppose in addition that there exist a locally free, R-flat sheaf $\sE$ on $X$ and an inclusion $\Omega_X^{[1]}(\textup{log}\,D) \subseteq \sE$ with torsion free cokernel.
Then there exist normal projective varieties $Y$ and $T$ as well as a finite cover $\gamma\colon Y \to X$ 
and a dominant rational map $a\colon Y \map T$ such that the following holds.
\begin{enumerate}
\item The morphism $\gamma$ is quasi-\'etale over $X \setminus \textup{Supp}\, D $. 
\item The variety $T$ is $\mathbb{Q}$-factorial and klt with $K_T\sim_\mathbb{Z}0$.
\item There exist open sets $T^\circ\subseteq T$ and $Y^\circ\subseteq Y$ with complement of codimension at least $2$ such that the map $a$ restricts to a projective morphism with rationally chain connected fibers $a^\circ\colon Y^\circ \to T^\circ$. Moreover, there is no $a$-exceptional divisor on $Y$.
\item There exist a locally free, R-flat sheaf $\sG$ on $T$ and an inclusion $\Omega_T^{[1]}\subseteq \sG$ with torsion free cokernel such that ${\gamma^*\sE}_{|Y^\circ}\cong (a^\circ)^*\sG_{|T^\circ}$.
\end{enumerate}
\end{prop}

We will need the following easy observation.

\begin{lemma}\label{lemma:saturated}
Let $X$ be a normal variety and let $\sE$ be a locally free sheaf on $X$. Let $\sG \subseteq \sE$ be a reflexive subsheaf. If the quotient sheaf $\sE/\sG$ is torsion free in codimension $1$, then $\sG$ is saturated in $\sE$.
\end{lemma}

\begin{proof}
By \cite[Proposition 1.1]{hartshorne80}, the saturation $\sG_1$ of $\sG$ in $\sE$ is reflexive. On the other hand, by assumption, $\sG$ and $\sG_1$ agree in codimension $1$. But this immediately implies that $\sG=\sG_1 \subseteq \sE$ (see \cite[Proposition 1.6]{hartshorne80}). 
\end{proof}

Before we give the proof of Proposition \ref{prop:main}, we need the following auxiliary results. 

\begin{lemma}\label{lemma:divisorial}
Let $X$ be a normal projective variety, and let $D$ be a reduced effective divisor on $X$. Suppose that $X$ is $\mathbb{Q}$-factorial with klt singularities. Let $\phi \colon X \to X_1$ be a divisorial Mori contraction with exceptional divisor $E$. Suppose furthermore that there exist a locally free, R-flat sheaf $\sE$ on $X$ and an inclusion $\Omega_X^{[1]}(\textup{log}\,D) \subseteq \sE$ with torsion free cokernel.
Then $E$ is contained in the support of $D$. 
\end{lemma}

\begin{proof}Recall that $E$ is an irreducible divisor since $X$ is $\mathbb{Q}$-factorial. We argue by contradiction and assume that $E$ is not contained in the support of $D$. By Theorem \ref{thm:descent}, there exists a locally free sheaf $\sE_1$ on $X_1$ such that $\sE\cong \phi^*\sE_1$. Let $\beta\colon Y\to X$ be a canonical resolution of $(X,D)$ and let $B$ be the largest reduced divisor contained in $\beta^{-1}\big(\textup{Supp}\,D\big)$. Let $F$ be the strict transform of $E$ in $Y$. By \cite[Lemma 3.6.2]{bchm}, 
$F$ is covered by curves $C$ contracted by $\phi\circ\beta$ such that $F\cdot C<0$. Suppose that $C \not\subset \textup{Supp}\, B$.
Then the composed map of locally free sheaves
$${\beta^*\sE^*}_{|C} \to {\beta^*T_X(-\textup{log}\,D)}_{|C} \cong {\beta^*\big(\beta_*T_Y(-\textup{log}\,B)\big)}_{|C} \to {T_Y(-\textup{log}\,B)}_{|C} \to {\sN_{F/Y}}_{|C}\cong{\sO_Y(-F)}_{|C}$$
is generically surjective. But ${\beta^*\sE^*}_{|C}$ is the trivial vector bundle while $F\cdot C <0$ by choice of $C$.
This yields a contradiction and shows that $E$ is contained in the support of $D$, completing the proof of the lemma.
\end{proof}

\begin{lemma}\label{lemma:fibration}
Let $X$ be a normal projective variety, and let $D$ be a reduced effective divisor on $X$ such that 
$(X,D)$ is log canonical and $K_X+D\sim_\mathbb{Q}0$. Suppose that $X$ is $\mathbb{Q}$-factorial with klt singularities.
Suppose in addition that there exist a locally free, R-flat sheaf $\sE$ on $X$ and an inclusion $\Omega_X^{[1]}(\textup{log}\,D) \subseteq \sE$ with torsion free cokernel.
Let $\psi \colon X \to Y$ be a Mori fiber space. 
Write $D=R+S$ where any irreducible component of $R$ (resp. $S$) maps onto $Y$ (resp. a proper subset of $Y$) and let $B$ be the $\mathbb{Q}$-divisor on $Y$ such that $S=\psi^* B$. Then there exists a finite cover $\tau \colon Y_1 \to Y$ such that the following holds. Let $X_1$ be the normalization of the product $Y_1\times_Y X$ with natural morphisms $\psi_1\colon X_1 \to Y_1$ and $\tau_1 \colon X_1 \to X$. Set $D_1:=\tau_1^*(K_X+D)-K_{X_1}$.
\begin{enumerate}
\item The morphism $\tau_1$ is quasi-\'etale.
\item Set $B_1:=\ulcorner\tau^* B\urcorner$. Then $B_1$ is the discriminant divisor of $(\psi_1,D_1)$. Moreover, $B_1$ is reduced and $K_{Y_1}+B_1\sim_\mathbb{Q}0$.
\item The pair $(Y_1,B_1)$ is log canonical and $Y_1$ is of klt type.
\item There exists a locally free, R-flat sheaf $\sG_1$ on $Y_1$ such that $\psi_1^*\sG_1\cong \tau_1^*\sE$, and an inclusion 
$\Omega_{Y_1}^{[1]}(\textup{log}\,B_1) \subseteq \sG_1$ with torsion free cokernel. 
\end{enumerate}
\end{lemma}

\begin{rem}
The existence of $B$ in the statement of Lemma \ref{lemma:fibration} is guaranteed by [Lemma 3.2.5 (2)]\cite{kmm}.
\end{rem}

\begin{proof}[Proof of Lemma \ref{lemma:fibration}]
Since $\psi$ is a Mori fiber space and $X$ is $\mathbb{Q}$-factorial by assumption, for any prime divisor $P$ on $Y$, $\textup{Supp}\,\psi^*(P)$ is irreducible. In particular, for any prime divisor $Q$ on $X$, $\psi(Q)$ has codimension at most $1$. In other words, there is no $\psi$-exceptional divisor on $X$.

By \cite[Theorem 1]{zhang_rcc}, the general fibers of $\psi$ are rationally connected. This in turn implies that any fiber of $\psi$ is rationally chain connected. Therefore, by Theorem \ref{thm:descent} applied to $\psi$, there exists a vector bundle $\sG$ on $Y$ such that $\sE\cong\psi^*\sG$. 
Notice that $\sG$ is R-flat by Lemma \ref{lemma:R_flat_functoriality}.

\medskip

We first compute the discriminant divisor of $(\psi,D)$. Let $P$ be a prime divisor on $Y$ which is not contained in the support of $B$ and write $\psi^*P=m Q$ where $m$ is a positive integer and $Q$ is a prime divisor. 
Let $Y^\circ \subseteq Y_{\textup{reg}}\setminus \textup{Supp}\, B$ be an open set such that $Y^\circ \cap P \neq \emptyset$, and 
set $X^\circ:=\psi^{-1}(Y^\circ)$. We may assume without loss of generality that there exists a finite morphism $\tau^\circ \colon Y_1^\circ \to Y^\circ$ with $Y_1^\circ$ smooth such that $\tau^\circ$ branches only over $P \cap Y^\circ$ with ramification index $m$. Let $X_1^\circ$ denotes the normalization of the fiber product $Y_1^\circ \times_Y X$ with natural morphisms $\psi_1^\circ\colon X_1^\circ \to Y_1^\circ$ and $\tau_1^\circ\colon X_1^\circ \to \psi^{-1}(Y^\circ)=:X^\circ$. Then $\tau_1^\circ$ is a quasi-\'etale morphism. Moreover, shrinking $Y^\circ$, if necessary, we may assume that $\psi_1^\circ$ has reduced fibers. 
Set $D_1^\circ:=(\tau_1^\circ)^*D_{|X^\circ}$, $\sE_1^\circ:=(\tau_1^\circ)^*\sE_{|X^\circ}$ and $\sG_1^\circ:=(\tau^\circ)^*\sG_{|X^\circ}$. Observe now (see Lemma \ref{lemma:pull_back_cover}) that
$$\Omega_{X_1^\circ}^{[1]}(\textup{log}\,D_1^\circ) \cong (\tau_1^\circ)^*\Omega_{X^\circ}^{[1]}(\textup{log}\,D_{|X^\circ}).$$
By Lemma \ref{lemma:saturated}, we have an inclusion  $\Omega_{X_1^\circ}^{[1]}(\textup{log}\,D_1^\circ) \subseteq \sE_1^\circ$ with torsion free cokernel. 
Notice that any irreducible component of $D_1^\circ$
dominates $Y_1^\circ$ by construction. Therefore, shrinking $Y^\circ$ again, if necessary, we may assume that the locus where the composed map
$$(\psi_1^\circ)^*(\sG_1^\circ)^*\cong(\sE_1^\circ)^*\to T_{X_1^\circ}(-\textup{log}\,D_1^\circ) \to T_{X_1^\circ} \to  (\psi_1^\circ)^* T_{Y_1^\circ}$$
is not surjective does not contain any fiber of $\psi_1$. This easily implies that the induced map of locally free sheaves $(\sE_1^\circ)^* \to T_{Y_1^\circ}$ is surjective.
Shrinking $Y^\circ$ further, we may assume that
$\sE^\circ$ is trivial and that $(\psi_1^\circ)^* T_{Y_1^\circ}$ is a direct summand of $T_{X_1^\circ}(-\textup{log}\,D_1^\circ)$.  
A classical result of complex analysis says that complex flows of vector fields on analytic spaces exist. This implies that the fibration
$(X_1^\circ,D_1^\circ)\to Y_1^\circ$ is locally trivial for the analytic topology.
One then readily checks that the discriminant divisor $B_Y$ of $(\psi,D)$ is
$$B_Y=\ulcorner B \urcorner +\sum \frac{m_P-1}{m_P}P$$ 
where $P$ runs through all prime divisors on $Y$ not contained in the support of $B$, and $m_P$ denotes the multiplicity of $\psi^*P$ along $\textup{Supp}\,\psi^*(P)$.

\medskip

By Corollary \ref{cor:cbf} applied to $\psi$, we must have $K_Y+B_Y \sim_\mathbb{Q} 0$. Let $\tau\colon Y_1 \to Y$ be the corresponding index one cover (see \cite[Section 2.4]{shokurov_log_flips}). 
Let $X_1$ be the normalization of the product $Y_1\times_Y X$ with natural morphisms $\psi_1 \colon X_1 \to Y_1$ and $\tau_1 \colon X_1 \to X$.
By construction, $\tau$ is \'etale at the generic points of $\textup{Supp}\, B$. Moreover, 
if $P$ is a prime divisor on $Y$ which is not contained in the support of $B$, 
then $\tau$ has ramification index $m_P$ along any irreducible component of $\tau^{-1}(P)$. 
This easily implies that $\tau_1$ is quasi-\'etale. 
Set $\sE_1:=\tau_1^*\sE$ and $D_1:=\tau_1^*D$. Notice that $\sE_1$ is R-flat and that we have an inclusion
$\Omega_{X_1}^{[1]}(\textup{log}\,D_1)\cong\tau^{[*]}\Omega_X^{[1]}(\textup{log}\,D) \subseteq \sE_1=\tau_1^*\sE$
with torsion free cokernel (see Lemma \ref{lemma:saturated}).
Notice also that the fibration $(\psi_1,D_1)$ is generically isotrivial.
Moreover, arguing as in the previous paragraph, we see that its discriminant is the reduced divisor $B_1=\ulcorner\tau^* B\urcorner$. By construction, we have $K_{Y_1}+B_1\sim_\mathbb{Q}0$. This proves Items (1) and (2).

\medskip

To prove Item (3), recall that $B_1$ is the discriminant of $(\psi_1,D_1)$. Set $D_i:=(1-\frac{1}{i})D$ and
$D_{1,i}:=\tau_1^*D_i$ for any integer $i\ge 1$. Notice that $D_{1,i}$ is effective.
Observe also that $(X,D_i)$ is klt for $i\ge 1$ since $X$ is $\mathbb{Q}$-factorial and klt and $(X,D)$ is log canonical by assumption. 
It follows that $(X_1,D_{1,i})$ is klt. 
Let $B_i$ (resp. $B_{1,i}$) be the discriminant divisor of $(\psi,D_i)$
(resp. $(\psi_1,D_{1,i})$). The $\mathbb{Q}$-divisor $K_Y+B_i$ is $\mathbb{Q}$-Cartier since $Y$ is $\mathbb{Q}$-factorial by \cite[Lemma 5.1.5]{kmm}.
By Fact \ref{fact:discriminant_finite_base_change}, we have $K_{Y_1}+B_{1,i} \sim_\mathbb{Q} \tau^*(K_Y+B_i)$ so that $K_{Y_1}+B_{1,i}$ is $\mathbb{Q}$-Cartier as well. On the other hand, since $K_X+D \equiv 0$ and $-K_X$ is $\psi$-ample, we see that $-(K_{X_1}+D_{1,i})$ is $\psi_1$-ample. 
By Lemma \ref{lemma:cbf_2} applied to $(\psi_1,D_{1,i})$, we conclude that $(Y_1,B_{1,i})$ is klt. 
Clearly, $B_{1,i} \le B_1$ and $B_{1,i} \to B_1$ as $i \to +\infty$. This proves Item (3). 

\medskip

Set $\sG_1:=\tau^*\sG$.
Note that $\sG_1$ is R-flat. By Lemma \ref{lemma:pull_back_fibration},
the standard pull-back map of K\"ahler differential induces an injective map
$$\Omega_{Y_1}^{[1]}(\textup{log}\,B_1) \to \big((\psi_1)_* \Omega_{X_1}^{[1]}(\textup{log}\,D_1)\big)^{**}\cong\sG_1.$$
By construction, $\psi_1$ has reduced fibers on some open set in $Y_1 \setminus \textup{Supp}\, B_1$ with complement of codimension at least $2$.
An easy local computation (see \cite[Chapitre 2, Proposition 3.2]{deligne_equ_diff}) then shows that
the cokernel of the inclusion $\Omega_{Y_1}^{[1]}(\textup{log}\,B_1) \subset \sG_1$ is torsion free in codimension $1$, and hence torsion free by Lemma \ref{lemma:saturated}.
This finishes the proof of the lemma
\end{proof} 

We are now in position to prove Proposition \ref{prop:main}.

\begin{proof}[Proof of Proposition \ref{prop:main}] We prove Proposition \ref{prop:main} by induction on $\dim X$.
 
\medskip 
 
If $\dim X =1$, then either $(X,D)\cong \big(\mathbb{P}^1,[0]+[\infty]\big)$, or $X$ is a Riemann surface of genus $1$ and 
$D=0$. So the statement holds true in this case.

\medskip

Suppose from now on that $\dim X \ge 2$.

\medskip

If $D=0$, then $K_X$ is torsion by assumption. Let $X_1 \to X$ be the corresponding index one cover, which is quasi-\'etale (\cite[Lemma 2.53]{kollar_mori}). Then $X_1$ is klt and $K_{X_1}\sim_\mathbb{Z}0$ by construction. So the statement holds true in this case.

\medskip

Suppose from now on that $D \neq 0$. Then $X$ is uniruled by \cite{miyaoka_mori_uniruledness} applied to a resolution of $(X,D)$.
Thus, we may run a minimal model program for $X$ and end with a Mori fiber space (see \cite[Corollary 1.3.3]{bchm}). There exists a sequence of maps

\begin{center}
\begin{tikzcd}[row sep=large, column sep=large]
X:=X_{1} \ar[r, "{\phi_{1}}", dashrightarrow] & \cdots \ar[r, "{\phi_{i-1}}", dashrightarrow] & X_{i} \ar[r, "{\phi_{i}}", dashrightarrow] & X_{i+1} \ar[r, "{\phi_{i+1}}", dashrightarrow] & \cdots \ar[r, "{\phi_{m-1}}", dashrightarrow] & X_{m} \ar[r, "{\psi}"] & Y
\end{tikzcd}
\end{center}

\noindent where the $\phi_{i}$ are either divisorial contractions or flips, and $\psi$ is a Mori fiber space. The spaces $X_{i}$ are normal, $\mathbb{Q}$-factorial, and $X_{i}$ has klt singularities for all $1\le i \le m$. 
Let $D_{i}$ be the push-forward of $D$ on $X_{i}$. 
By \cite[Lemma 3.2.5 (2)]{kmm}, we have $K_{X_{i}}+D_{i}\sim_\mathbb{Q} 0$. Moreover, the pair $(X_{i},D_{i})$ is log canonical by \cite[Lemma 3.38]{kollar_mori}. 

We define projective varieties $W_i$ inductively for any integer $1 \le i \le m$ as follows. Set $W_m:=X_m$.  
Let $W_i$ be a resolution of the graph of the rational map $X_i \map W_{i+1}$ for $1 \le i \le m-1$ with morphism $p_i\colon W_i \to X_i$. Let $q_i\colon W_1 \to W_i$ be the natural morphism. 
We then show inductively that there exist locally free, R-flat sheaves $\sE_i$ on $X_i$ and an inclusion 
$\Omega_{X_{i}}^{[1]}(\textup{log}\,D_{i}) \subseteq \sE_{i}$ with torsion free cokernel
such that $p_1^*\sE \cong (p_i\circ q_i)^*\sE_i$.
Set $\sE_1:=\sE$.

Suppose that $\phi_i$ is a divisorial contraction. By Theorem \ref{thm:descent}, there exists a locally free sheaf $\sE_{i+1}$ on $X_{i+1}$ such that $\sE_i\cong \phi_i^*\sE_{i+1}$. Moreover, $\sE_{i+1}$ is R-flat by Lemma \ref{lemma:R_flat_functoriality}.

Suppose now that $\phi_i$ is the flip of a small extremal contraction $c_i\colon X_i \to Z_i$, and let $c_{i+1}\colon X_{i+1} \to Z_i$ be the natural morphism.
By Theorem \ref{thm:descent} again, there exists a locally free sheaf $\sG_i$ on $Z_i$ such that $\sE_i \cong c_i^*\sG_i$. Set $\sE_{i+1}:=c_{i+1}^*\sG_i$. Then $\sG_i$ is R-flat by Lemma \ref{lemma:R_flat_functoriality} and hence so is $\sE_{i+1}$.

In either case, the inclusion $\Omega_{X_i}^{[1]}(\textup{log}\,D_i) \subseteq \sE_i$ induces an inclusion 
$\Omega_{X_{i+1}}^{[1]}(\textup{log}\,D_{i+1}) \subseteq \sE_{i+1}$ with torsion free cokernel 
by Lemma \ref{lemma:saturated}. Moreover, one readily checks that $p_1^*\sE \cong (p_{i+1}\circ q_{i+1})^*\sE_{i+1}$.

Then, we apply Lemma \ref{lemma:fibration} to $\psi$. Write $D_m=R+S$ where any irreducible component of $R$ (resp. $S$) maps onto $Y$ (resp. a proper subset of $Y$) and let $B$ be the $\mathbb{Q}$-divisor on $Y$ such that $S=\psi^* B$ (see \cite[Lemma 3.2.5 (2)]{kmm}). There exists a finite cover $\tau \colon Y_1 \to Y$ such that the following holds. Let $X_{m+1}$ be the normalization of the product $Y_1\times_Y X_m$ with natural morphisms $\psi_1\colon X_{m+1} \to Y_1$ and $\tau_1 \colon X_{m+1} \to X_m$. 
%Set $D_{m+1}:=\tau_1^*(K_{X_m}+D)-K_{X_{m+1}}$.
\begin{itemize}
\item The morphism $\tau_1$ is quasi-\'etale.
\item The divisor $B_1:=\ulcorner\tau^* B\urcorner$ is reduced and $K_{Y_1}+B_1\sim_\mathbb{Q}0$.
\item The pair $(Y_1,B_1)$ is log canonical and $Y_1$ is of klt type.
\item There exists a locally free, R-flat sheaf $\sG_1$ on $Y_1$ such that $\psi_1^*\sG_1\cong \tau_1^*\sE$, and an inclusion 
$\Omega_{Y_1}^{[1]}(\textup{log}\,B_1) \subseteq \sG_1$ with torsion free cokernel. 
\end{itemize}

Let $\beta \colon Y_2 \to Y_1$ be a $\mathbb{Q}$-factorialization of $Y_1$, whose existence is established in 
\cite[Corollary 1.37]{kollar_kovacs_singularities}. Recall that $\beta$ is a small birational projective morphism and that 
$Y_2$ is $\mathbb{Q}$-factorial with klt singularities.
Set $B_2:=(\beta^{-1})_*B_1$. Then $\beta^*\sG_1$ is obviously R-flat and there is an inclusion 
$\Omega_{Y_2}^{[1]}(\textup{log}\,B_2)\subseteq \beta^*\sG_1$ with torsion free cokernel (see Lemma \ref{lemma:saturated}). Moreover, $(Y_2,B_2)$ is log canonical by \cite[Lemma 3.10]{kollar97} and $K_{Y_2}+B_2\sim_\mathbb{Q}0$.

The induction hypothesis applied to $(Y_2,B_2)$ asserts that there exist normal projective varieties $Z_2$ and $T$ as well as a finite cover $\eta\colon Z_2 \to Y_2$ and a dominant rational map $a\colon Z_2 \map T$ such that the following holds.
\begin{itemize}
\item The morphism $\eta$ is quasi-\'etale over $Y_2 \setminus \textup{Supp}\, B_2 $. 
\item The variety $T$ is $\mathbb{Q}$-factorial and klt with $K_T\sim_\mathbb{Z}0$.
\item There exist open sets $T^\circ\subseteq T$ and $Z_2^\circ\subseteq Z_2$ with complement of codimension at least $2$ such that the map $a$ restricts to a projective morphism with rationally chain connected fibers $a^\circ\colon Z_2^\circ \to T^\circ$. Moreover, there is no $a$-exceptional divisor on $Z_2$.
\item There exist a locally free, R-flat sheaf $\sG$ on $T$ and an inclusion $\Omega_T^{[1]}\subseteq \sG$ with torsion free cokernel such that ${(\eta\circ\beta)^*\sG_1}_{|Z_2^\circ}\cong (a^\circ)^*\sG_{|T^\circ}$.
\end{itemize}
Let $Z_1$ be the normalization of the product $X_{m+1}\times_{Y_1} Z_2$ and let $Z$ be the 
normalization of $X_1$ in the function field of $Z_1$. We have a diagram as follows:

\begin{center}
\begin{tikzcd}[row sep=large, column sep=large]
Z\ar[r, "{\textup{birational}}", dashrightarrow]\ar[ddd, "{\gamma, \textup{ finite}}"]\ar[rrr, "{b}", bend left, dashrightarrow] & Z_1\ar[dd, "\textup{generically} \\ \textup{finite}"{align=left,font=\scriptsize}]\ar[r, "{\textup{fibration}}"] & Z_2\ar[d, "{\eta\textup{, finite}}"]\ar[r, "{a}"', dashrightarrow] & T\\
&&Y_2\ar[d, "{\beta,\textup{ small birational}}"] &\\
& X_{m+1}\ar[d, "{\tau_1,\textup{ quasi-\'etale}}"]\ar[r, "{\psi_1}"] & Y_1\ar[d, "{\tau,\textup{ finite}}"] & \\
X=X_1\ar[r, "{\textup{birational}}"', dashrightarrow] & X_m \ar[r, "{\psi,\textup{ MFS}}"'] &  Y. &
\end{tikzcd}
\end{center}

We now show that $\gamma$ is quasi-\'etale over $X \setminus \textup{Supp}\, D$. Let $P$ be a prime divisor on $X$ and suppose that $\gamma$ is ramified over $P$. If $P$ is contracted by the birational map $X \map X_m$ then $P$ is contained in the support of $D$ by Lemma \ref{lemma:divisorial}.
Suppose that $P$ is not contracted by $X \map X_m$ and let $P_m$ be the push-forward of $P$ on $X_m$. Then the generically finite map $Z_1 \to X_{m+1}$ is ramified over $\tau_1^*P_m$. On the other hand, since $\psi$ is a Mori fiber space, any irreducible component of $\tau_1^*P_m$ maps onto a closed subset of codimension at most $1$ in $Y_1$.
Thus, since $\eta$ is quasi-\'etale over $Y_2 \setminus \textup{Supp}\, B_2 $ and $\beta$ is a small birational contraction, any irreducible component of $\tau_1^*P_m$ must be contained in the support of $\psi_1^*B_1$. This shows that $P$ is contained in the support of $D$.

\medskip

Notice that $T$ has canonical singularities since $K_T\sim_\mathbb{Z}0$. This implies that
$T$ is not uniruled. 
By assumption, $X$ is $\mathbb{Q}$-factorial with klt singularities and $(X,D)$ is log canonical. Applying Lemma \ref{lemma:klt_singularities_finite_cover}, we see that $Z$ is of klt type. This in turn implies that 
the rational map $b\colon Z \map T$ is almost proper by \cite[Corollary 1.7]{hacon_mckernan}.
By \cite[Corolary 1.8]{hacon_mckernan} together with \cite[Theorem 1.2]{hacon_mckernan}, the general fibers of $\psi \colon X_m \to Y$ are rationally connected. This implies that the general fibers of the rational map $Z \map Z_2$ are rationally connected as well. Since the general fibers of $a$ are also rationally connected by \cite[Corolary 1.8]{hacon_mckernan}, applying \cite[Corollary 1.3]{ghs03}, we see that the general fibers of $b\colon Z \map T$ are rationally connected. Therefore, $b\colon Z \map T$ is the maximally rationally chain connected fibration of $Z$.

Set $G:=\gamma^*(K_X+D)-K_Z$. By Lemma \ref{lemma:klt_singularities_finite_cover}, $G$ is effective since $\gamma$ is quasi-\'etale over $X\setminus\textup{Supp}\,D$. Moreover, the pair $(Z,G)$ is log canonical, and $\gamma^{-1}\big(\textup{Supp}\, D\big) \subseteq \textup{Supp}\,G$.

Let $Q$ be a prime divisor on $Z$. Suppose that $Q$ is $b$-exceptional. If $Q$ is contracted by the rational map $Z \map X_m$ then $Q$ must be contained in the support of $G$ by Lemma \ref{lemma:divisorial}. On the other hand, by Proposition \ref{prop:zhang} below, any irreducible component of $G$ dominates $T$, yielding a contradiction. Therefore, $Q$ is not contracted by the rational map $Z \map X_m$. Then, since $\psi$ is a Mori fiber space, $Q$ is not exceptional for $Z \map Z_2$. It follows that the image of $Q$ on $Z_2$ is $a$-exceptional. But this contradicts the induction hypothesis. This shows that there is no $b$-exceptional divisor $Z$.

Suppose now that there is a uniruled prime divisor $P$ on $T$. Set $k:=\dim T$. The inclusion $\Omega_T^{[1]}\subseteq \sG$ then yields an inclusion $$s\colon\sO_T(K_T)\cong \sO_T\subseteq \wedge^k\sG$$
with torsion free cokernel at a general point of $P$.
Let $\nu \colon \mathbb{P}^1 \to X$ be a rational curve passing through a general point $p \in P$. By general choice of $\nu$, we may assume that  $s(p)\neq 0$. On the other hand, $\nu^*(\wedge^k\sG)$ is numerically flat by assumption. This immediately implies that $s$ is nowhere vanishing along $\nu(\mathbb{P}^1)$. By \cite[Lemma 1.20]{demailly_peternell_schneider94}, the reflexive sheaf $\Omega_T^{[1]}$ is then locally free in a neighbourhood of $\nu(\mathbb{P}^1)$. This in turn implies that $T$ is smooth along $\nu(\mathbb{P}^1)$ by the solution of the Zariski-Lipman conjecture for log canonical spaces (see \cite[Theorem 1.1]{druel_zl} or \cite[Corollary 1.3]{graf_kovacs_zl}).
Then $\nu^*T_T \cong \sO_{\mathbb{P}^1}^{\oplus k}$, yielding a contradiction since $T_{\mathbb{P}^1} \subset \nu^*T_T$.
By \cite[Corollary 1.7]{hacon_mckernan}, there exist open sets $T^\circ\subseteq T$ and $Z^\circ\subseteq Z$ such that $b$ restricts to a projective morphism with connected fibers $b^\circ\colon Z^\circ \to T^\circ$ and such that $T^\circ$ has complement of codimension at least $2$. Since there is no $b$-exceptional divisor on $Z$, $Z^\circ$ has complement of codimension at least $2$. 

Finally, one readily checks that $(b^\circ)^*\sG_{T^\circ}\cong {f^*\sE}_{|Z^\circ}$, finishing the proof of the proposition.
\end{proof}

\begin{prop}[{\cite[Main Theorem and Remark 1]{zhang_mrc_fibration}}]\label{prop:zhang}
Let $(X,D)$ be a log canonical pair with $X$ projective and $D$ effective, and let $f \colon X \map T$ be the maximally rationally chain connected fibration. Suppose that $-(K_X+D)$ is nef. Then $f$ is semistable in codimension $1$. Moreover, any irreducible component of $D$ dominates $T$. 
\end{prop}

\section{Proofs}

In this section we prove our main results. Note that Theorem \ref{thm:main_intro} is an immediate consequence of Theorem \ref{thm:main} below.

\begin{thm}\label{thm:main}
Let $X$ be a normal projective variety of klt type, and let $D$ be a reduced effective divisor on $X$ such that $(X,D)$ is log canonical.
Suppose that the sheaf $\Omega_X^{[1]}(\textup{log}\,D)$ is locally free and R-flat, and that $-(K_X+D)$ is nef.
Then there exist a smooth projective variety $T$ with $K_T\equiv 0$ as well as a surjective morphism with connected fibers 
$a\colon X \to T$. The fibration $(X,D)\to T$ is locally trivial for the analytic topology and any fiber $F$ of the map $a$ is a toric variety with boundary divisor $D_{|F}$. Moreover, $T$ contains no rational curve.
\end{thm}

\begin{proof} 
For the reader's convenience, the proof is subdivided into a number of steps. 

\medskip

\noindent\textit{Step 1.}
By the cone theorem for log canonical spaces (see \cite[Theorem 1.4]{fujino_non_vanishing}), we must have $K_X+D\equiv 0$ since $-(K_X+D)$ is nef and $(K_X+D)\cdot C=0$ for any rational curve $C \subset X$ by assumption.

Let $\beta\colon X_1 \to X$ be a canonical resolution of $(X,D)$ and let $D_1$ be the largest reduced divisor contained in $\beta^{-1}\big(\textup{Supp}\,D\big)$. By Lemma \ref{lemma:canonical_resolution}, we have $\Omega_{X_1}^1(\textup{log}\,D_1) \cong \beta^*\Omega_X^1(\textup{log}\,D)$. In particular, $\Omega_{X_1}^1(\textup{log}\,D_1)$ is locally free and R-flat.
Moreover, $(X_1,D_1)$ is log canonical and $K_{X_1}+D_1\equiv 0$. Applying \cite[Theorem 0.1]{CKP_numerical}, we see that $K_{X_1}+D_1$ is torsion. 

Suppose now that the conclusion of Theorem \ref{thm:main} holds for the pair $(X_1,D_1)$. We show that the conclusion of Theorem \ref{thm:main} also holds for the pair $(X,D)$. 
By assumption, there exist a smooth projective variety $T_1$ with $K_{T_1}\equiv 0$ as well as a surjective morphism with connected fibers 
$a_1\colon X_1 \to T_1$. The fibration $(X_1,D_1)\to T_1$ is locally trivial for the analytic topology and any fiber $F_1$ of the map $a_1$ is a smooth toric variety with boundary divisor ${D_1}_{|F_1}$. Moreover, $T_1$ contains no rational curve. Observe also that any irreducible component of $D_1$ maps onto $T_1$.

By \cite[Theorem 1.2]{hacon_mckernan}, every fiber of $\beta$ is rationally chain connected. On the other hand, $T_1$ contains no rational curve by assumption. It follows that the rational map $a_1 \circ\beta^{-1}$ is a morphism $a \colon X \to T_1=:T$. 
By Theorem \ref{thm:main_descent_intro} applied to $a_1$ together with the projection formula, there exists a vector bundle $\sE$ on $T$ such that $\Omega_{X}^{[1]}(\textup{log}\,D) \cong a^*\sE$. Since $(X_1,D_1)$ is locally trivial over $T_1$ (for the analytic topology), the
morphism $\sE^* \to T_T=T_{T_1}$ induced by the composed map
$$a_1^*\sE^* \cong T_{X_1}(-\textup{log}\,D_1) \to T_{X_1} \to a_1^* T_{T_1}$$
is surjective. This easily implies that the composed morphism
$$a^*\sE^* \cong T_X(-\textup{log}\,D) \to T_X \to a^* T_T$$
is surjective as well since $\beta^*T_X(-\textup{log}\,D) \cong T_{X_1}(-\textup{log}\,D_1)$.
Moreover, $T_T$ is locally a direct summand of $\sE^*$.
Now, a classical result of complex analysis says that complex flows of vector fields on analytic spaces exist. This implies that the fibration $(X_1,D_1)\to T$ is locally trivial for the analytic topology.

Let $F$ be any fiber of $a$ and let $F^\circ \subseteq F$ be the open set where $(F,D_{|F})$ is log smooth. Note that $(F,D_{|F})$ is log canonical. In particular, $F^\circ$ has complement of codimension at least $2$ in $F$. The sequence
$$ 0 \to T_{F^\circ}(-\textup{log}\, {D}_{|F^\circ}) \to T_{X}(-\textup{log}\, D)_{|F^\circ}\cong \sO_{F^\circ}^{\oplus \dim X} \to  {\sN_{F/X}}_{|F^\circ}\cong \sO_{F^\circ}^{\oplus \dim T} \to 0$$
is exact (see \cite[Lemma 3.2]{druel_zl}), and hence $T_{F}(-\textup{log}\, {D}_{|F})\cong \sO_{F}^{\oplus \dim F}$ since both sheaves are reflexive and agree on $F^\circ$. Let $\mu\colon F_1 \to F$ be a canonical resolution, and let $D_{F_1}$ be the largest reduced divisor contained in $\mu^{-1}\big(\textup{Supp}\,D_{|F}\big)$. By Lemma \ref{lemma:canonical_resolution}, 
$$\Omega_{F_1}^{1}(\textup{log}\,D_{F_1}) \cong \mu^* \Omega_F^{[1]}(\textup{log}\,D_{|F})\cong \sO_{F_1}^{\oplus \dim F_1}.$$
By \cite[Corollary 1]{winkelmann_log_trivial}, $F_1$ is a toric variety with boundary divisor $D_{F_1}$. 
This in turn implies that $F$ is a toric variety with boundary divisor $D_{|F}$.
This shows that the conclusion of Theorem \ref{thm:main} holds for the pair $(X,D)$. 
Therefore, we may assume without loss of generality that the following holds.

\begin{assumption}
The pair $(X,D)$ is log smooth and $K_X+D\sim_\mathbb{Q} 0$. 
\end{assumption}

We prove Theorem \ref{thm:main} by induction on $\dim X$.
 
\medskip 
 
If $\dim X =1$, then either $(X,D)\cong \big(\mathbb{P}^1,[0]+[\infty]\big)$, or $X$ is a Riemann surface of genus $1$ and 
$D=0$. The statement holds true in this case.

\medskip

Suppose from now on that $\dim X \ge 2$ and apply Proposition \ref{prop:main}. There exist normal projective varieties $X_1$ and $T_1$ as well as a finite cover $\gamma\colon X_1 \to X$ 
and a dominant rational map $a_1\colon X_1 \map T_1$ such that the following properties hold. 
\begin{itemize}
\item The morphism $\gamma$ is quasi-\'etale over $X \setminus \textup{Supp}\, D $. 
\item The variety $T_1$ is $\mathbb{Q}$-factorial and klt with $K_{T_1}\sim_\mathbb{Z}0$.
\item There exist open sets $T_1^\circ\subseteq T_1$ and $X_1^\circ\subseteq X_1$ with complement of codimension at least $2$ such that the map $a_1$ restricts to a projective morphism with rationally chain connected fibers $a_1^\circ\colon X_1^\circ \to T_1^\circ$. Moreover, there is no $a_1$-exceptional divisor on $X_1$.
\item There exist a locally free, R-flat sheaf $\sG_1$ on $T_1$ and an inclusion $\Omega_{T_1}^{[1]}\subseteq \sG_1$ with torsion free cokernel such that ${\gamma^*\Omega_X^{[1]}(\textup{log}\,D)}_{|X_1^\circ}\cong (a_1^\circ)^*{\sG_1}_{|T_1^\circ}$.
\end{itemize}

Set $D_1:=\gamma^*(K_X+D)-K_{X_1}$ and $D_1^\circ:={D_1}_{|X_1^\circ}$. By Lemma \ref{lemma:pull_back_cover}, we have
$$\gamma^*\Omega_X^{[1]}(\textup{log}\,D) \cong \Omega_{X_1}^{[1]}(\textup{log}\,D_1).$$

\noindent\textit{Step 2.} We show that the fibration $(X_1^\circ,{D_1}_{|X_1^\circ})\to T_1^\circ$ is locally trivial for the analytic topology, and that any fiber $F_1$ of $a_1^\circ$ is a toric variety with boundary divisor ${D_1}_{|F_1}$.
We may assume without loss of generality that $T_1^\circ$ is contained in the smooth locus of $T_1$.
Recall from Proposition \ref{prop:zhang} that any irreducible component of $D_1$ maps onto $T_1$.
By Proposition \ref{prop:zhang} again, we may also assume that $a_1$ has reduced fibers over $T_1^\circ$. It follows that the composed map 
$$(a_1^\circ)^*({\sG_1^*}_{|T_1^\circ})\cong T_{X_1^\circ}(-\textup{log}\,D_1^\circ) \to T_{X_1^\circ} \to (a_1^\circ)^*T_{T_1^\circ}$$
is surjective. Therefore, the induced map ${\sG_1^*}_{|T_1^\circ}\to T_{T_1^\circ}$ is surjective as well, and hence
$T_{T_1^\circ}$ is locally a direct summand of ${\sG_1^*}_{|T_1^\circ}$. As before, this implies that
the fibration $(X_1^\circ,D_1^\circ)\to T_1^\circ$ is locally trivial for the analytic topology.

Let $F_1$ be a general fiber of $a_1$, and set $D_{F_1}:={D_1}_{|F_1}$. 
One readily checks that 
$T_{F_1}(-\textup{log}\,D_{F_1})\cong \sO_{F_1}^{\oplus \dim F_1}$. Let $\mu_1\colon F_2 \to F_1$ be a canonical resolution, and let $D_{F_2}$ be the largest reduced divisor contained in $\mu_1^{-1}\big(\textup{Supp}\,D_{F_1}\big)$. By Lemma \ref{lemma:canonical_resolution}, 
$$\Omega_{F_2}^{1}(\textup{log}\,D_{F_2}) \cong \mu_1^* \Omega_{F_1}^{[1]}(\textup{log}\,D_{F_1})\cong \sO_{F_2}^{\oplus \dim F_2}.$$
It follows that $F_2$ is a toric variety with boundary divisor $D_{F_2}$ (see \cite[Corollary 1]{winkelmann_log_trivial}).
This in turn implies that $F_1$ is a toric variety with boundary divisor $D_{F_1}$, completing the proof of the claim.

\medskip

\noindent\textit{Step 3.} Let $\beta_1\colon X_2 \to X_1$ be a canonical resolution of $(X_1,D_1)$ and let $D_2$ be the largest reduced divisor contained in $\beta_1^{-1}\big(\textup{Supp}\,D_1\big)$. By Lemma \ref{lemma:canonical_resolution}, we have $\Omega_{X_2}^1(\textup{log}\,D_2) \cong \beta^*\Omega_{X_1}^1(\textup{log}\,D_1)$. In particular, $\Omega_{X_2}^1(\textup{log}\,D_2)$ is locally free and R-flat.
Moreover, $(X_2,D_2)$ is log canonical and $K_{X_2}+D_2\sim_\mathbb{Q} 0$. Set $a_2:=a_1 \circ\beta_1^{-1}\colon X_2 \map T_1$.
We have a commutative diagram as follows:
\begin{center}
\begin{tikzcd}[row sep=large, column sep=large]
X_2\ar[d, "{\beta_1}"']\ar[r, "{a_2}", dashrightarrow] & T_1\ar[d, equal]\\
X_1\ar[d, "{\gamma}"']\ar[r, "{a_1}", dashrightarrow] & T_1\\
X. &
\end{tikzcd}
\end{center}
By Proposition \ref{prop:zhang}, any irreducible component of $D_2$ maps onto $T_1$. In particular, there is no $a_2$-exceptional divisor on $X_2$.
Set $X_2^\circ:=(a_2^\circ)^{-1}(T_1^\circ)=\beta_1^{-1}(X_1^\circ)$ and $D_2^\circ:={D_2}_{|X_2^\circ}$.
Arguing as in Step 2 above, we see that, shrinking $T_1^\circ$ if necessary, we may assume that the fibration
$(X_2^\circ,D_2^\circ)$ is locally trivial over $T_1^\circ$. Moreover, a general fiber $F_2$ of $a_2^\circ:={a_2}_{|X_2^\circ}$ is a smooth toric variety with boundary divisor ${D_2}_{|F_2}$. We may also assume without loss of generality that $X_2^\circ$ has complement of codimension at least $2$ since there is no $a_2$-exceptional divisor on $X_2$. 

Let $C_2$ be an irreducible component of $D_2$. 
The short exact sequence (see \cite[Lemma 2.13.2]{kebekus_kovacs_invent})
$$0 \to \sO_{C_2} \to {T_{X_2}(-\textup{log}\, D_2)}_{|C_2} \to T_{C_2}(-\textup{log}\, (D_2-C_2)_{|C_2}) \to 0$$
implies that $T_{C_2}(-\textup{log}\, (D_2-C_2)_{|C_2})$ is R-flat. Moreover, $K_{C_2}+(D_2-C_2)_{|C_2}\sim_\mathbb{Q}0$.
The induction hypothesis applied to $C_2$ then says that there exists a smooth projective variety $T_2$ with $K_{T_2}\equiv 0$ as well as a smooth morphism with rational fibers $b_2\colon C_2 \to T_2$. Moreover, the fibration $\big(C_2,(D_2-C_2)_{|C_2}\big)\to T_2$ is locally trivial for the analytic topology and $T_2$ contains no rational curve.

Notice that fibers of $C_2\cap X_2^\circ \to T_1^\circ$ are projective with rational connected components.
On the other hand, the number of connected components of $C_2 \cap a_2^{-1}(t)$
does not depend on the point $t \in T_1^\circ$ since $(X_2^\circ,D_2^\circ) \to T_1^\circ$ is a locally trivial fibration. Let $T_3^\circ \to T_1^\circ$ be the corresponding \'etale cover, and let $T_3$ be the normalization of $T_1$ in the function field of $T_1^\circ$. Note that $T_3 \to T_1$ is quasi-\'etale. In particular, $T_3$ is klt and $K_{T_3}\equiv 0$.
Moreover, there is a birational map $\iota\colon T_3\map T_2$.
Since $T_2$ contains no rational curve and $T_3$ is klt, $\iota$ is a morphism by \cite[Corollary 1.7]{hacon_mckernan}. On the other hand,
since $T_2$ is smooth and both $T_2$ and $T_3$ have numerically trivial canonical class, we conclude that $\iota$ is an isomorphism. In particular, $T_3$ is smooth and contains no rational curve.

Replacing $T_1$ by $T_3$ and $X_1$ by a quasi-\'etale cover, if necessary, we may assume without loss of generality that the following holds.

\begin{assumption}
The variety $T_1$ is smooth and contains no rational curve.
\end{assumption}

\medskip

\noindent\textit{Step 4.} Then \cite[Corollary 1.7]{hacon_mckernan} implies that both $a_1$ and $a_2$ are morphisms.
By the Nagata-Zariski purity theorem, there is an \'etale morphism $\eta\colon T_2 \to T_1$ such that 
$\eta\circ b_2 = {a_2}_{|C_2}$. 
Now, we have a commutative diagram as follows:
\begin{center}
\begin{tikzcd}
{a_2^*\sG_1^*\cong T_{X_2}(-\textup{log}\,D_2)}_{|C_2}\ar[d, twoheadrightarrow]\ar[r] & {T_{X_2}}_{|C_2}\ar[r] & {a_2^*T_{T_1}}_{|C_2}\ar[d, equal] \\
T_{C_2}(-\textup{log}\, (D_2-C_2)_{|C_2}) \ar[rr, twoheadrightarrow] & & {a_2^*T_{T_1}}_{|C_2}\cong ({a_2}_{|C_2})^*T_{T_1}
\end{tikzcd}
\end{center}
It follows that the map $\sG_1^* \to T_{T_1}$ induced by the composed map 
$$a_2^*\sG_1^*\cong T_{X_2}(-\textup{log}\,D_2) \to T_{X_2} \to a_2^*T_{T_1}$$
is surjective. Arguing as above, this implies that the fibration $(X_2,D_2)\to T_1$ is locally trivial for the analytic topology.
By Step 1, we see that the fibration $(X_1,D_1)\to T_1$ is also locally trivial.

\medskip

\noindent\textit{Step 5.}
Let $f \colon X \map R$ be the maximally rationally chain connected fibration. Recall that $f$ is an almost proper map and that its general fibers are rationally connected. Let $\omega \in H^0(X,\Omega_X^q\otimes\sL)$ be a twisted $q$-form defining the foliation $\sH$ on $X$ induced by $f$. By Proposition \ref{prop:zhang}, any irreducible component of $D$ maps onto $R$. This implies that the zero set of the reflexive pull-back $\omega_1 \in H^0\big(X_1,\Omega_{X_1}^{[q]}\otimes\gamma^*\sL\big)$ of $\omega$ has codimension at least $2$ (see for instance \cite[Lemma 3.4]{cd1fzerocan}). Moreover, $\omega_1$ obviously defines the foliation $\sH_1$ on $X_1$ induced by the map $a_1$.
It follows that $\gamma^*\sL\cong \det \sN_{\sH_1}\cong\sO_{X_1}$. 
Let $\omega_2 \in H^0(X_2,\Omega_{X_2}^{q})$ be the pull-back of $\omega$ to $X_2$. Then $\omega_2$ defines the foliation $\sH_2$ on $X_2$ induced by the map $a_2$. Since $\det \sN_{\sH_2}\cong\sO_{X_2}$ and $\sH_2$ is regular, we conclude that $\omega_2$ is nowhere vanishing. This in turn implies that $\sH$ is regular.

By Lemma \ref{lemma:smoothness} below, $f$ extends to a smooth morphism with rationally connected fibers $a \colon X \to T$ onto a smooth projective variety $T$. By Theorem \ref{thm:main_descent_intro}, there is a locally free, R-flat sheaf $\sG$ on $T$ such that 
$a^*\sG\cong \Omega_X^1(\textup{log}\,D)$.
Arguing again as above, we conclude that the fibration $(X,D)\to T$ is locally trivial for the analytic topology and that any fiber $F$ of $a$ is a smooth toric variety with boundary divisor $D_{|F}$. 

Notice that $K_T$ is torsion since $\gamma^*\sL\cong \sO_{X_1}$. In particular, $T$ is not uniruled. This easily implies that the image on $X$ of any fiber of $a_1$ is contracted by $a$. On the other hand, if $F$ is a general fiber of $a$, then any connected component of 
$(\gamma\circ\beta)^{-1}(F)$ is rationally connected. This immediately implies that $\dim T = \dim T_1$. Moreover, by the rigidity lemma, there is a finite morphism $\tau \colon T_1 \to T$ such that $a \circ\gamma = \tau\circ a_1$. Since $T$ is smooth and both $T$ and $T_1$ have numerically trivial canonical class, we conclude that $\tau$ is \'etale. It follows that $T$ contains no rational curve. 
This completes the proof of the theorem.
\end{proof}

The following result is an immediate consequence of \cite[Corollary 2.11]{split_tangent_rc}.

\begin{lemma}\label{lemma:smoothness}
Let $X$ be a complex projective manifold, and let $f\colon X \map Y$ be an almost proper dominant rational map onto a normal projective variety $Y$. Suppose that the general fibers of $f$ are rationally connected. Suppose furthermore that the foliation 
$\sG$ on $X$ induced by $f$ is regular. Then $f$ is a smooth morphism. In particular, $Y$ is smooth.
\end{lemma}

\begin{proof}[{Proof of Corollary \ref{cor:main_intro_1}}] By Theorem \ref{thm:main_intro}, there exist a smooth projective variety $T$ with $K_T\equiv 0$ as well as a smooth morphism with connected fibers $a\colon X \to T$. The fibration $(X,D)\to T$ is locally trivial for the analytic topology and any fiber $F$ of the map $a$ is a smooth toric variety with boundary divisor $D_{|F}$. Since $\pi_1(X)=\{1\}$ by assumption, we also have $\pi_1(T)=\{1\}$. It follows that
$K_T\sim_\mathbb{Z}0$. This in turn implies that $h^{p,0}(X)\ge 1$ where $p:=\dim T$. Therefore, we must have $p=0$, proving the corollary.
\end{proof}

\begin{proof}[{Proof of Corollary \ref{cor:main_intro_2}}] Notice that 
$T_X(-\textup{log}\, D)$ is $R$-flat. Moreover, $K_X+D \equiv 0$, so that Theorem \ref{thm:main_intro} applies.
There exist a smooth projective variety $T$ with $K_T\equiv 0$ as well as a smooth morphism with connected fibers $a\colon X \to T$. The fibration $(X,D)\to T$ is locally trivial for the analytic topology and any fiber $F$ of the map $a$ is a smooth toric variety with boundary divisor $D_{|F}$. By Theorem \ref{thm:main_descent_intro}, there exists a vector bundle $\sG$ on $T$ such that $a^*\sG \cong T_X(-\textup{log}\, D)$. Set $p:=\dim T$. By Lemma \ref{lemma:pull_back_fibration}, there is a nonzero map
$\wedge^p\sG \to \sO_T(-K_T)$. Notice that both $\wedge^p\sG$ and $\sO_T(-K_T)$ are numerically flat. Thus,
applying \cite[Proposition 1.16]{demailly_peternell_schneider94}, we see that the morphism of locally free sheaves $\sG \to T_T$ is surjective. This in turn implies that $T_T$ is numerically flat.
By \cite[Corollary 1.19]{demailly_peternell_schneider94}, we have $c_1(T)=0$ and $c_2(T)=0$. As a classical consequence of Yau's theorem on the existence of a 
K\"ahler-Einstein metric, $T$ is then covered by a complex torus (see \cite[Chapter IV Corollary 4.15]{kobayashi_diff_geom_vb}). This finishes the proof of the corollary.
\end{proof}

\begin{proof}[{Proof of Proposition \ref{prop:main_intro_3}}] Note that all Chern classes of $T_X(-\textup{log}\, D)$ vanish. By \cite[Theorem 1.3]{campana_paun19} (see also \cite[Theorem 4]{schnell_CP}), $T_X(-\textup{log}\, D)$ is slope-semistable with respect to any polarization. Then \cite[Corollary 3.10]{simpson_higgs_flat} implies that $T_X(-\textup{log}\, D)$ is numerically flat so that Corollary \ref{cor:main_intro_2} applies, proving the corollary. 
\end{proof}

%\bibliographystyle{amsalpha}
%\bibliography{foliation}

\providecommand{\bysame}{\leavevmode\hbox to3em{\hrulefill}\thinspace}
\providecommand{\MR}{\relax\ifhmode\unskip\space\fi MR }
\providecommand{\MRhref}[2]{%
  \href{http://www.ams.org/mathscinet-getitem?mr=#1}{#2}
}
\providecommand{\href}[2]{#2}

\end{document}